\documentclass[a4paper,11pt]{article} \usepackage[T1]{fontenc}
\usepackage{graphicx}
\usepackage{mathrsfs}
\usepackage{verbatim}
\usepackage{float}
\usepackage{tablefootnote}
\usepackage{bm}
\usepackage[utf8]{inputenc} 
\usepackage[english]{babel} 
\usepackage{braket}
\usepackage{amsmath}
\usepackage{amssymb}
\usepackage{amsthm}
\usepackage{multirow}
\usepackage{threeparttable}
\usepackage{tikz-cd}
\usepackage{pgfplots}
\pgfplotsset{compat=1.18}
\usetikzlibrary{fit,shapes.misc}
\usepackage{xcolor}

\definecolor{mygreen}{RGB}{0,170,0}
\definecolor{myyellow}{RGB}{255,190,0}
\definecolor{myorange}{RGB}{255,105,35}
\usepackage{array}
\usepackage{longtable}
\usepackage{caption}
\usepackage[most]{tcolorbox}
\usepackage[toc,page]{appendix}
\usepackage{hyperref}
\hypersetup{
	colorlinks=true,
	linkcolor=blue,      
	urlcolor=blue,
	citecolor=blue
}
\newcolumntype{C}[1]{>{\centering\arraybackslash}p{#1}}
\newcolumntype{P}[1]{>{\centering\arraybackslash}p{#1}}

\DeclareMathOperator{\NE}{NE}
\DeclareMathOperator{\Nef}{Nef}

\DeclareMathOperator{\Eff}{Eff}

\DeclareMathOperator{\Exc}{Exc}
\DeclareMathOperator{\Ext}{Ext}
\DeclareMathOperator{\dom}{dom}
\DeclareMathOperator{\Bl}{Bl}

\DeclareMathOperator{\length}{length}
\DeclareMathOperator{\Mov}{Mov}
\theoremstyle{plain}
\newtheorem{proposition}{Proposition}[section]
\newtheorem{theorem}[proposition]{Theorem}
\newtheorem{lemma}[proposition]{Lemma}
\newtheorem{corollary}[proposition]{Corollary}
\theoremstyle{definition}
\newtheorem{definition}[proposition]{Definition}
\newtheorem{remark}[proposition]{Remark}

\newtheorem*{notations}{Notation}

\theoremstyle{plain}
\newcommand{\thistheoremname}{}
\newtheorem*{genericthm}{\thistheoremname}

\usepackage{titlesec}

\title{On $\mathbb{P}^1$-bundles on $\mathbb{P}^3$ that are small modifications of a Fano fourfold}
\author{Pier Roberto Pastorino}

\titleformat{\section}         
{\normalfont\large\bfseries}  
{\thesection}{1em}{}

\titleformat{\subsection}
{\normalfont\normalsize\bfseries}
{\thesubsection}{1em}{}

\titleformat{\subsubsection}
{\normalfont\small\bfseries}
{\thesubsubsection}{1em}{}
\begin{document}
	\maketitle
	\begin{abstract}
		We find necessary conditions for the projectivization of a rank two vector bundle $\mathcal{E}$ on $\mathbb{P}^3_{\mathbb{C}}$ to be related to a smooth Fano 4-fold by a small $\mathbb{Q}$-factorial modification (SQM). Furthermore, we fully describe the birational geometry of $\mathbb{P}_{\mathbb{P}^3}(\mathcal{E})$ when $\mathcal{E}$ is stable, with $c_1=-1$, $c_2=2$, starting from the description of the moduli space of such bundles \cite{HS81}. In this case, we prove that there exists a flip $X^+\dasharrow\mathbb{P}_{\mathbb{P}^3}(\mathcal{E})$, where $X^+$ is smooth and Fano. We also determine the birational geometry and numerical invariants of $X^+$.  
	\end{abstract}
	\tableofcontents
	\section{Introduction}
	Fano varieties of the form $\mathbb{P}_{\mathbb{P}^3}(\mathcal{E})$, where $\mathcal{E}$ is a rank two vector bundle on $\mathbb{P}^3_{\mathbb{C}}$ have been completely classified in \cite{SW90}. In this paper, we consider the problem of classifying rank-two vector bundles $\mathcal{E}$ on $\mathbb{P}^3_{\mathbb{C}}$ whose projectivization is not Fano but admits a small $\mathbb{Q}$-factorial modification (SQM) to a smooth Fano 4-fold; that is, such that there exists a birational map
	\(
	\mathbb{P}_{\mathbb{P}^3}(\mathcal{E})\dasharrow X^+
	\)
	that is an isomorphism in codimension one, where $X^+$ is a smooth Fano 4-fold.

	Our main theorem states that a necessary condition for $\mathcal{E}$ to be such a vector bundle is that $\mathcal{E}$ be stable and, up to normalization, have
	\[
	c_1=-1,\qquad c_2\in\{2,4,6\}
	\]
	(Theorem \ref{TheoremNonFanoCase}(a)).
	
	On the other hand, we prove that if $\mathcal{E}$ is stable, $c_1=-1$, and $c_2=2$, then $\mathbb{P}_{\mathbb{P}^3}(\mathcal{E})$ is indeed a small $\mathbb{Q}$-factorial modification of a smooth Fano 4-fold (Theorem \ref{TheoremNonFanoCase}(b)). Our study of this case relies on the description of the moduli space of stable vector bundles on $\mathbb{P}_{\mathbb{C}}^3$ with the above Chern classes given by R. Hartshorne and I. Sols \cite{HS81}. In their paper, the authors study the zero locus of a nonzero section of \(\mathcal E(1)\), and use it to obtain a description of certain sections of \(\mathcal E(2)\) with particularly useful geometric properties. They prove that the moduli space of such vector bundles is an irreducible nonsingular rational variety of dimension $11$ and admits a universal family \cite[Theorem 3.1]{HS81}.  
	This study also determines the set of jumping lines for $\mathcal{E}$. The birational geometry of $\mathbb{P}_{\mathbb{P}^3}(\mathcal{E})$, however, remains largely unexplored. In this paper, we provide a detailed description of its Mori cone, of the Fano 4-fold obtained from it by an SQM, and of the SQM itself, which consists of one flip. We show that all projective bundles of this form are related by SQMs to a Fano 4-fold $X^+$ admitting two extremal contractions: one small and one divisorial of type $(3,1)$ (see Theorem \ref{TheoremBir}). Moreover, we have
	\[
	c_1(X^+)^4=289,\qquad
	c_1(X^+)^2c_2(X^+)=166,\qquad
	h^0(-K_{X^+})=63,\qquad
	h^{2,2}(X^+)=3,
	\]
	and
	\[
	h^{i,j}(X^+)=0\qquad\text{for all }i\ne j.
	\]
	We compared these invariants with those appearing in several existing databases of Fano 4-folds and found no matches.
	
	The cases $c_2=4$ and $c_2=6$ appear to be substantially more difficult from the point of view of birational geometry, and a thorough study of $\mathbb{P}_{\mathbb{P}^3}(\mathcal{E})$ in these cases may be the subject of future work. For instance, although a description of the moduli space of stable vector bundles with $c_1=-1$ and $c_2=4$ is available \cite{BM85}, its geometry appears to be considerably richer and more varied than in the case $c_2=2$.

	\paragraph{Notation and Conventions.}
	We work over the field of complex numbers. Let $X$ be a smooth projective variety of arbitrary dimension. 
	
	A \emph{small $\mathbb{Q}$-factorial modification} (SQM) is a birational map $X\dashrightarrow \tilde{X}$, where $\tilde{X}$ is a normal, projective and $\mathbb{Q}$-factorial variety, that is an isomorphism in codimension one.
	
	$N_1(X)$ (respectively $N^1(X)$) denotes the real vector space of one-cycles (respectively, Cartier divisors) with real coefficients, modulo numerical equivalence, and $\dim N_1(X)=\dim N^1(X)=\rho_X$ is the Picard number of $X$. 
	We denote by $\Nef(X)$ the nef cone of $X$, by $\Mov(X)$ the movable cone of $X$ and by $\Eff(X)$ the effective cone of $X$.
	
	The symbol $\sim$ stands for linear equivalence of divisors, while $\equiv$ stands for numerical equivalence of one-cycles, as well as divisors. 
	
	$\NE(X)\subset N_1(X)$ is the convex cone generated by the classes of effective curves. An \emph{extremal} ray $R$ of $\NE(X)$ is a one dimensional ray such that if $a,b\in \NE(X)$, we have $a+b\in R \iff a,b\in R$. Note that we do not require that $R$ is $K$-negative for it to be called \emph{extremal}.
	
	A \emph{contraction} is a surjective morphism $\varphi:X\to Y$ with connected fibers, where $Y$ is normal and projective. A contraction is \emph{extremal} if it contracts curves whose classes in $\NE(X)$ belong to an extremal ray $R$. 
	
	By \emph{flip}, we refer to the flip of a small extremal contraction. 
	
	Recall that
	\[
	H^2(\mathbb P^3;\mathbb Z)\simeq\mathbb Z,\qquad
	H^4(\mathbb P^3;\mathbb Z)\simeq\mathbb Z,
	\]
	with generators $\mathit{h}=c_1(\mathcal O_{\mathbb P^3}(1))$ and $\mathit{h}^2$, respectively. If $\mathcal E$ is a vector bundle on $\mathbb P^3$, we denote by $c_i(\mathcal{E})\in H^{2i}(\mathbb{P}^3;\mathbb{Z})$ its $i$-th Chern class, and we denote by $c_1$ and $c_2$ the unique integers such that
	\[
	c_1(\mathcal E)=c_1\mathit{h},\qquad c_2(\mathcal E)=c_2\mathit{h}^2.
	\]
	We will sometimes confuse the notations of $c_i(\mathcal{E})$ and $c_i$, when there is no possible ambiguity.

	Moreover, if $X=\mathbb{P}_{\mathbb{P}^3}(\mathcal{E})$, we always denote its tautological divisor by $\xi$, or $\xi_{\mathcal{E}}$, and we denote by $H$ the pull-back in $X$ of a plane in $\mathbb{P}^3$ through the projection $\pi\colon \mathbb{P}_{\mathbb{P}^3}(\mathcal{E})\to \mathbb{P}^3$. 
	
	We denote by $\mathcal{E}(m)$ the tensor product $\mathcal{E}\otimes\mathcal{O}_{\mathbb{P}^3}(m)$, for any $m\in\mathbb{Z}$. Recall that we have $\xi_{\mathcal{E}(m)}=\xi +mH$ when we identify $\mathbb{P}_{\mathbb{P}^3}(\mathcal{E}(m))\simeq\mathbb{P}_{\mathbb{P}^3}(\mathcal{E})$, for any $m\in\mathbb{Z}$.      
	
	Recall that the splitting type of the restriction of a vector bundle \(\mathcal E\) on \(\mathbb{P}^n\) to a general line is constant. A line \(L\subset\mathbb{P}^n\) is called a \emph{jumping line} for \(\mathcal{E}\) if the splitting type of \(\mathcal{E}|_L\) differs from that of the restriction to a general line.
	
	\paragraph{Acknowledgements.}
	I am very grateful to my advisor, Cinzia Casagrande, for her constant guidance and support. ChatGPT-6 Astra helped me identify and correct minor mistakes in Section 4.2 that were present in the first version of this paper. The author is a member of GNSAGA-INdAM. 
	
	\section{The main theorem}
	We are going to prove the following theorem:  
	
	\begin{theorem}\label{TheoremNonFanoCase}
		Let $\mathcal{E}$ be a rank 2 vector bundle on $\mathbb{P}^3$, and $X:=\mathbb{P}(\mathcal{E})$.
		We have the following:
		\begin{itemize}
			\item[(a)] If $X$ is not Fano, but related to a smooth Fano 4-fold by a SQM, then, up to normalization, $\mathcal{E}$ is a stable vector bundle, with $c_1=-1$ and $c_2\in\{2,4,6\}$. 
			\item[(b)] If $\mathcal{E}$ is a stable vector bundle with $c_1=-1$ and $c_2=2$, then $X$ is obtained from a smooth Fano 4-fold by a flip. 
		\end{itemize}
	\end{theorem}
	
	We now briefly recall the definition and basic properties of stable vector bundles on $\mathbb{P}^n$ with rank $2$.
	
	\begin{definition}
		A vector bundle $\mathcal{E}$ of rank $2$ on $\mathbb{P}^n$ is \emph{stable} if for every invertible subsheaf $\mathcal{L}$ of $\mathcal{E}$,
		\begin{align*}
			c_1(\mathcal{L})<\frac{1}{2} c_1(\mathcal{E}).
		\end{align*} 
	\end{definition}
	
	\begin{remark}[\cite{Har78}, Remark 3.0.1]
		There is a well known elementary characterization of stable vector bundles that we are going to use in the next paragraphs. Let $\mathcal{E}$ be a rank $2$ vector bundle on $\mathbb{P}^n$ that is \emph{normalized}, i.e. $c_1(\mathcal{E})\in\{-1,0\}$. Then
		\begin{align*}
			\mathcal{E}\text{ is stable }\iff H^0(\mathcal{E})=0.
		\end{align*}
		
	\end{remark}
	
	This paper is devoted to providing a proof of Theorem \ref{TheoremNonFanoCase}, after going through some necessary auxiliary results. In Section 3 and Section 4 we will prove Theorem \ref{TheoremNonFanoCase} (a) and (b) respectively. 
	
	Moreover, in case (b), we also provide a full description of the birational geometry of $X$, as follows: 
	
	\begin{theorem}\label{TheoremBir}
		In case (b) of Theorem \ref{TheoremNonFanoCase}, let $\xi$ be the tautological divisor, and $H$ the pullback in $X$ of a plane in $\mathbb{P}^3$. Then 
		\begin{align*}
			\Nef(X)&=\mathbb{R}_{\ge 0}[H]+\mathbb{R}_{\ge 0}[\xi+3H].\\
			\Mov(X)&=\mathbb{R}_{\ge 0}[H]+\mathbb{R}_{\ge 0}[\xi+2H].\\
			\Eff(X)&=\mathbb{R}_{\ge 0}[H]+\mathbb{R}_{\ge 0}[\xi+H].
		\end{align*}  
		There exists a flip $\phi\colon X\dashrightarrow X^+$, where $X^+$ is a smooth Fano 4-fold that has two extremal contractions $\gamma_1\colon X^+\to Z_1$, $\gamma_2\colon X^+\to Z_2$ such that 
		\begin{itemize}
			\item $\gamma_1$ is a small contraction whose exceptional locus is a single \emph{exceptional plane} (in the sense of Lemma \ref{LemmaSQMFano}).
			\item $\gamma_2$ is a divisorial contraction that on the exceptional divisor is a quadric bundle onto a smooth curve $C$ in a variety $Z_2$ that locally embeds as a hypersurface in a smooth 5-fold, so that $C$ is locally a complete intersection in $\mathbb{C}^5$ \cite[Main Theorem (4q)]{Tak99}. All fibers of $\gamma_2$ are smooth, except for exactly two of them.
		\end{itemize}
		Moreover,  $c_1(X^+)^4=289$, $c_1(X^+)^2c_2(X^+)=166$,  $h^0(-K_{X^+})=63$, $h^{2,2}(X^+)=3$, and $h^{i,j}(X^{+})=0$ for all $i\ne j$.  
	\end{theorem}
	\begin{center}
		\begin{tikzpicture}[line cap=round, line join=round, scale=0.85]
			
			\definecolor{mygreen}{RGB}{0,150,70}
			\definecolor{myyellow}{RGB}{230,175,0}
			\definecolor{myorange}{RGB}{230,95,35}
			
			\draw[thick] (0,0) -- (9.5,0);
			\draw[thick] (0,0) -- (0,8.5);
			
			
			\fill[myorange, opacity=0.12]
			(0,0) -- (9.4,0)
			arc[start angle=0, end angle=45, radius=9.4]
			-- cycle;
			
			\fill[myyellow, opacity=0.20]
			(0,0) -- (8.9,0)
			arc[start angle=0, end angle=26.565, radius=8.9]
			-- cycle;
			
			\fill[mygreen, opacity=0.22]
			(0,0) -- (8.4,0)
			arc[start angle=0, end angle=18.435, radius=8.4]
			-- cycle;
			
			\draw[very thick, mygreen]
			(8.4,0) arc[start angle=0, end angle=18.435, radius=8.4];
			
			\draw[very thick, myyellow]
			(8.9,0) arc[start angle=0, end angle=26.565, radius=8.9];
			
			\draw[very thick, myorange]
			(9.4,0) arc[start angle=0, end angle=45, radius=9.4];
			
			\draw[thick] (0,0) -- (8.5,8.5);
			\draw[thick] (0,0) -- (9.5,4.75);
			\draw[thick] (0,0) -- (9.5,3.17);
			
			\fill (0,1.4) circle (0.075);
			\fill (1.4,0) circle (0.075);
			
			\fill (1.4,1.4) circle (0.075);
			\fill (2.8,1.4) circle (0.075);
			\fill (4.2,1.4) circle (0.075);
			
			\fill (7.0,2.8) circle (0.075);
			
			\tikzset{
				graph label/.style={ font=\small,
					inner sep=2pt
				}
			}
			
			\node[graph label, left] at (0,1.9) {$[\xi]$};
			\node[graph label] at (1.4,-0.45) {$[H]$};
			\node[graph label] at (1.4,1.9) {$[\xi+H]$};
			\node[graph label, right] at (2.4,1.9) {$[\xi+2H]$};
			\node[graph label, right] at (4.2,1.9) {$[\xi+3H]$};
			\node[graph label, right] at (7.0,3) {$[-K_X]$};
			
			\draw[thick, mygreen] (9.5,7.5) -- (10.15,7.5);
			\node[right] at (10.25,7.5) {$\Nef(X)$};
			
			\draw[thick, myyellow] (9.5,6.5) -- (10.15,6.5);
			\node[right] at (10.25,6.5) {$\Mov(X)$};
			
			\draw[thick, myorange] (9.5,5.5) -- (10.15,5.5);
			\node[right] at (10.25,5.5) {$\Eff(X)$};
			
		\end{tikzpicture}
	\end{center}

	\section{Necessary conditions on $\mathcal{E}$ so that $X=\mathbb{P}(\mathcal{E})$ is SQM of a Fano}
	
	This section is devoted to the proof of Part (a) of Theorem \ref{TheoremNonFanoCase}. We begin by stating two lemmas that will be crucial for that goal. 
	
	\begin{lemma}\label{LemmaSQMFano}
		Let $\tilde{Z}$ be a smooth Fano 4-fold and $\phi\colon Z\dashrightarrow \tilde{Z}$ a SQM. We have the following:
		\begin{itemize}
			\item[(a)] $Z$ is smooth and $Z\setminus\dom(\phi)=\ell_1\cup\dots\cup\ell_r$, where $\ell_i$ are pairwise disjoint smooth rational curves such that $\mathcal{N}_{\ell_i/Z}\simeq\mathcal{O}_{\mathbb{P}^1}(-1)^{\oplus 3}$ for all $i\in\{1,\dots r\}$. Note that $-K_Z\cdot\ell_i=-1$. Such curves are called \emph{exceptional lines}.
			
			Moreover $\tilde{Z}\setminus \dom(\phi^{-1})=L_1\cup\dots\cup L_r$ where $L_i$ are pairwise disjoint surfaces such that $L_i\simeq\mathbb{P}^2$ and $\mathcal{N}_{{L_i}/\tilde{Z}}\simeq \mathcal{O}_{\mathbb{P}^2}(-1)^{\oplus 2}$. Such surfaces are called \emph{exceptional planes}.

			In addition, $\phi$ factors as
			\begin{equation*} 
				\begin{tikzcd}
					Z\ar[rr,dashrightarrow,"\phi",bend left]&\hat{Z}\ar[l,"\sigma"]\ar[r,"\sigma^{+}"']&\tilde{Z}
				\end{tikzcd}
			\end{equation*}
			where $\sigma$ is the blow-up of $\ell_1\cup\dots\cup\ell_{r}$ and $\sigma^{+}$ is the blow-up of $L_1\cup\dots\cup L_r$. 
			\item[(b)] Let $C\subset Z$ be an irreducible curve, different from $\ell_1,\dots\ell_r$, and intersecting $\ell_1\cup\dots\cup\ell_r$ in $s\ge 0$ points; then $-K_{Z}\cdot C\ge 1+s$. If $C$ is smooth, then this inequality also holds when $s$ accounts for multiplicities. 
			\item[(c)] If $-K_{Z}\cdot C=1$, then $C\cap(\ell_1\cup\dots\cup\ell_r)=\emptyset$. 
		\end{itemize} 
	\end{lemma}
	\begin{proof}
		Remark 3.6 in \cite{Cas13} proves $(a)$, $(c)$; $(b)$ is proved under the assumption that $s$ is the number of points in the intersection $C\cap (\ell_1\cup\dots\ell_r)$ without counting multiplicities.
		However, under the assumption that $C$ is smooth, the same proof of Remark 3.6 in \cite{Cas13} yields the inequality $-K_X\cdot C\ge 1+s$, where $s$ accounts for multiplicities.  
	\end{proof}
	
	\begin{definition}
		Let $X$ be a smooth 4-dimensional variety. We define an \emph{exceptional line} in $X$ as a smooth rational curve $\ell$ in $X$ such that $\mathcal{N}_{\ell/X}\simeq\mathcal{O}_{\ell}(-1)^{\oplus 3}$. 
	\end{definition}
	
	\begin{remark}
		By adjunction, if $\ell$ is an \emph{exceptional line} in a smooth 4-fold $X$, then $-K_X\cdot \ell=-1$. Moreover, as a consequence of Lemma \ref{LemmaSQMFano}, the existence of a finite number of disjoint exceptional lines in $X$ is a necessary condition for $X$ to be a non-Fano variety related to a Fano 4-fold by an SQM.
	\end{remark}
	
	\begin{lemma}[\cite{CS24}, Remark 2.5]\label{LemmaRmk2.5}
		Let $Y$ be a smooth quasi-projective variety, $X:=\mathbb{P}_{Y}(\mathcal{E})$ where $\mathcal{E}$ is a rank two vector bundle on $Y$, and $\pi\colon X\to Y$ the $\mathbb{P}^1$-bundle. 
		
		Let $C\subset Y$ be a smooth projective rational curve and set $S:=\pi^{-1}(C)\simeq\mathbb{F}_e$, where $\mathbb{F}_e$ is the Hirzebruch surface and $e\ge 0$. Let $\Gamma^{-}$ and $\Gamma^{+}$ be respectively the negative section and a positive section of $\pi|_S\colon S\to C$ (namely $(\Gamma^{-})^2=-e$ and $(\Gamma^{+})^2=e$ in $S$). Then $\det(\mathcal{E})\cdot C\equiv e$ mod $2$ and 
		\begin{align*}
			-K_X\cdot \Gamma^{-}=-K_Y\cdot C-e,\text{ }-K_X\cdot\Gamma^{+}=-K_Y\cdot C+e
		\end{align*} 
	\end{lemma}
	
	Consider the following set of assumptions:
	\begin{itemize}
		\item[($\blacktriangle$)] $\mathcal{E}$ is a rank two vector bundle on $\mathbb{P}^3$ such that $X=\mathbb{P}(\mathcal{E})$ is related to a smooth Fano 4-fold by a SQM, and $X$ is not Fano.
	\end{itemize}
	
	\begin{remark}
		Notice that under the assumption ($\blacktriangle$), $X$ is a Mori Dream Space. In particular, its Mori cone is rational polyhedral. 
	\end{remark}
	
	\begin{lemma}\label{LemmaFe}
		Assume ($\blacktriangle$). Then $c_1(\mathcal{E})$ is odd and for every exceptional line $\mathbf{l}\subset X$, its image $\pi(\mathbf{l})$ is a smooth rational curve in $\mathbb{P}^3$.
		
		Moreover, let $C$ be a smooth rational curve in $\mathbb{P}^3$ with degree $d$, and $S:=\pi^{-1}(C)\subset X$. Then $S\simeq \mathbb{F}_e$ where $e$ is a natural number such that $e\equiv d$ mod $2$, and the following holds:
		\begin{itemize}
			\item If $S$ contains an exceptional line, then  $e=4d+1\ge 5$, and $d$ is odd. Moreover, the exceptional line must coincide with the negative section $\Gamma^{-}$ of $S$. 
			\item Otherwise, $e\le 4d-1$
		\end{itemize} 
	\end{lemma}
	\begin{proof}
		By Lemma \ref{LemmaRmk2.5}, $S\simeq \mathbb{F}_e$ with $\det(\mathcal{E})\cdot C\equiv e$ mod 2. Let $\Gamma^{+}$ and $\Gamma^{-}$ denote respectively a positive section and the negative section of $\mathbb{F}_e$, namely $(\Gamma^{+})^2=e$, $(\Gamma^{-})^2=-e$. We have $-K_{X}\cdot \Gamma^{+}=-K_{\mathbb{P}^3}\cdot C+e$ and $-K_{X}\cdot \Gamma^{-}=-K_{\mathbb{P}^3}\cdot C-e$.  
		
		If $S$ contains an exceptional line $\mathbf{l}$, let $\alpha,\beta\in\mathbb{Z}_{\ge 0}$ be coefficients such that $\mathbf{l}\sim \alpha \Gamma^{-}+\beta \mathbf{f}$, where $\mathbf{f}$ is a fiber of the $\mathbb{P}^1$-bundle $S\to \mathbb{P}^1$. Recall that, by definition, we have $-K_X\cdot \mathbf{l} =-1$. It follows that $\alpha\ne 0$, because $-K_X\cdot \beta\mathbf{f}=2\beta\ge 0$. Suppose, by contradiction, that $\mathbf{l}\ne\Gamma^{-}$. By Proposition 2.21 \cite{Har77}, $\alpha=1$, $\beta\ge 1$, or $\alpha\ge 2$, $\beta\ge \alpha e$.  Observe that $-K_X\cdot \mathbf{l} = (4d-e)\alpha+2\beta=4d\alpha+(2\beta-\alpha e)$. If $\alpha\ge 2$, it follows that $-K_X\cdot \mathbf{l}\ge 4d\alpha>0$, a contradiction. Therefore $\alpha=1$ and $\beta\ge 1$. However, in that case we would have $-K_X\cdot \Gamma^{-}=4d-e<4d-e+2\beta=-K_X\cdot \mathbf{l}=-1$, a contradiction, because there cannot be any irreducible curves in $X$ with anticanonical degree less than $-1$ (see Lemma \ref{LemmaSQMFano}). Therefore, $\mathbf{l}$ must coincide with $\Gamma^{-}$ in $S$, and we have $-1=-K_{X}\cdot\Gamma^{-}=4d-e$, from which it follows $e=4d+1$. In particular $e$ is odd, and so must be the intersection number $\det(\mathcal{E})\cdot C$, from which it follows that both $c_1(\mathcal{E})$ and $d$ must also be odd.
		
		Observe that by assumption and by Lemma \ref{LemmaSQMFano} (a), an exceptional line $\mathbf{l}$ in $X$ does exist. Moreover, if $f$ is a fiber of $\pi\colon X\to \mathbb{P}^3$, then by Lemma \ref{LemmaSQMFano} (b), $f$ can meet $\mathbf{l}$ in a number of points $s$ (counted with multiplicities) that satisfies $2=-K_X\cdot f\ge1+s$, so that $s\in\{0,1\}$. This implies that the degree of $\pi|_{\mathbf{l}}\colon\mathbf{l}\to \pi(\mathbf{l})$ is one and at every point $p\in \mathbf{l}$, the differential $d(\pi|_{\mathbf{l}})_p$ is nonzero, so that in particular $\pi(\mathbf{l})$ is a smooth rational curve. Therefore, the arguments used in the previous paragraph can be applied to $S=\pi^{-1}(\pi(\mathbf{l}))$, and we can conclude that $c_1(\mathcal{E})$ is indeed odd in this setting.
		
		On the other hand, if $S$ does not contain any exceptional line, then $1\le -K_{X}\cdot \Gamma^{-}=-K_{\mathbb{P}^3}\cdot C-e=4d-e$, hence $e\le 4d-1$. Moreover, since we just proved that $c_1(\mathcal{E})$ is odd, and $\det(\mathcal{E})\cdot C\equiv e$ mod $2$, we have that $d\equiv e$ mod $2$.

	\end{proof}

	Under the assumptions ($\blacktriangle$), denote by $H$ the pullback in $X$ of a hyperplane in $\mathbb{P}^3$, and denote by $\mathbf{f}$ a fiber of $\pi\colon X\to \mathbb{P}^3$. As we proved in Lemma \ref{LemmaFe}, each exceptional line in $X$ is \emph{horizontal}, meaning that its image in $\mathbb{P}^3$ through $\pi$ is a smooth irreducible curve. Among all exceptional lines $\mathbf{l}$ in $X$, fix one that realizes the minimal value of $\deg(\pi(\mathbf{l}))$. We denote it by $\mathbf{l}_0$, and we denote by $C_0$ its image $\pi(\mathbf{l}_0)$. Since $-K_{X}\cdot \mathbf{l}_0=-1\ne 2$, $\mathbf{l}_0$ is not a fiber of $\pi$ and $C_0$ is a rational curve of degree $d\ge 1$.
	
	We then have that the space $N^1(X)$ is generated by $H$ and $-K_{X}$, and $N_1(X)$ is generated by $[\mathbf{l}_0]$ and $[\mathbf{f}]$, with the following intersection table:
	
	\begin{center}
		\begin{tabular}{|c|c|c|}
			\hline
			$\cdot$ & $-K_{X}$ & $H$ \\
			\hline
			$\mathbf{l}_0$ & $-1$ & $d$\\
			\hline
			$\mathbf{f}$ & $2$ & $0$\\
			\hline
		\end{tabular}
	\end{center}
	
	\begin{lemma}
		Under the assumptions $(\blacktriangle)$, we have
		\[\NE(X)=\mathbb{R}_{\ge 0}[\mathbf{l}_0]+\mathbb{R}_{\ge 0}[\mathbf{f}].\]
	\end{lemma}	
	\begin{proof}
		Let $Y$ be any irreducible curve in $X$. Clearly $N_1(X)=\mathbb{R}[\mathbf{l}_0]+\mathbb{R}[\mathbf{f}]$, so there exist $\alpha,\beta\in\mathbb{R}$ such that $Y\equiv\alpha\mathbf{l}_0+\beta\mathbf{f}$. Now, $H$ is nef, so $0\le H\cdot Y=\alpha d$, which implies $\alpha\ge 0$. If $Y$ is not an exceptional line, then $1\le-K_X\cdot Y=-\alpha+2\beta$, which implies $\beta\ge \frac{1}{2}(1+\alpha)\ge 0$, so $[Y]\in\mathbb{R}_{\ge 0}[\mathbf{l}_0]+\mathbb{R}_{\ge 0}[\mathbf{f}]$. 
		
		On the other hand, if $Y$ coincides with an exceptional line $\mathbf{l}$, then $-1=-K_X\cdot Y=-\alpha+2\beta$, which means $\beta=\frac{1}{2}(\alpha-1)$. 
		Recall that $\deg(\pi(Y))=H\cdot Y=\alpha d\ge 1$. Now, if $0< \alpha<1$, then $\deg(\pi(Y))=\alpha d<d=\deg(C_0)$, a contradiction by the minimality of $C_0$. Hence $\alpha\ge 1$ and $\beta=\frac{1}{2}(\alpha-1)\ge 0$, so we can conclude that $[Y]\in \mathbb{R}_{\ge 0}[\mathbf{l}_0]+\mathbb{R}_{\ge 0}[\mathbf{f}]$.
	\end{proof} 
	
	Without loss of generality (by Lemma \ref{LemmaFe}), in this section we can and will always assume $c_1(\mathcal{E})=-1$. Again, denote $C_0=\pi(\mathbf{l}_0)$. Denote by $S_0$ the preimage $\pi^{-1}(C_0)$. We know that there are $m,n\in \mathbb{Z}$ such that $\mathcal{E}|_{C_0}\simeq \mathcal{O}_{C_0}(m)\oplus\mathcal{O}_{C_0}(n)$, and $S_0\simeq \mathbb{P}(\mathcal{E}|_{C_0})$. Suppose $m\ge n$ and call $e:=m-n$. By Lemma \ref{LemmaFe}, $\mathbf{l}_0$ coincides with $\Gamma^{-}$ in $S_0\simeq\mathbb{F}_e$, and $e=4d+1$. Moreover $m+n=\det(\mathcal{E})\cdot C_0=-d\equiv e$ mod $2$, so that $e$ and $d$ are both odd. We also get
	\begin{align*}
		\begin{cases}
			m-n=4d+1\\
			m+n=-d
		\end{cases}
	\end{align*}
	so that $m=\frac{3d+1}{2}$, $n=-\frac{5d+1}{2}$. 
	
	Now, consider $\xi$ the tautological divisor of $\mathbb{P}(\mathcal{E})$. Let $\xi_{S_0}$ be the tautological divisor of $\mathbb{P}(\mathcal{E}|_{C_0})$ and $\mathbf{f}$ be the class of a fiber of $S_0\to C_0$. We have
	\begin{equation*}
		\xi|_{S_0}=\xi_{\mathcal{E}|_{C_0}}=\xi_{\mathcal{E}(-m)|_{C_0}}+m\mathbf{f}=\Gamma^{-}+m\mathbf{f}=\mathbf{l}_0+m\mathbf{f}. 
	\end{equation*}
	Then $\xi\cdot \mathbf{l}_0=\xi_{S_0}\cdot \mathbf{l}_0= {\mathbf{l}_0}^2+m\mathbf{l}_0\cdot \mathbf{f}= -e+m=n=-\frac{1}{2}(5d+1)$. We then have the following intersection table for $X$:
	\begin{center}
		\begin{tabular}{|c|c|c|}
			\hline
			$\cdot$ & $\xi$ & $H$ \\
			\hline
			$\mathbf{l}_0$ & $-\frac{1}{2}(5d+1)$ & $d$\\
			\hline
			$\mathbf{f}$ & $1$ & $0$\\
			\hline
		\end{tabular}
	\end{center}
	We are ready to prove the following lemma.
	
	\begin{lemma}\label{LemmaBoundc2}
		Assume ($\blacktriangle$). If $c_1(\mathcal{E})=-1$, then $c_2(\mathcal{E})$ is an even integer that is at most equal to 6. 
	\end{lemma}
	
	\begin{proof}
		Consider $T$ a nef and big divisor such that $T\cdot \mathbf{l}_0=0$ (the pullback of the generator of the Picard group of the target of the small contraction that contracts $\mathbf{l}_0$). Write it as a linear combination of $\xi$ and $H$: $T\equiv\alpha \xi+\beta H$. Then 
		\begin{align*}
			0=T\cdot \mathbf{l}_0=-\frac{1}{2}(5d+1)\alpha+d\beta  
		\end{align*} 
		thus we can take $\alpha=2d$ and $\beta=5d+1$ (we take $\alpha>0$ because $T\cdot \mathbf{f}>0$).
		
		As a consequence of the well known relation $\xi^2-\pi^{*}c_1\xi+\pi^{*}c_2=0$, we deduce that 
		\begin{align*}
			H^4=0;\quad H^3\xi=1;\quad H^2\xi^2=c_1=-1;\quad H\xi^3=-c_2+c_1^2=-c_2+1;\quad \xi^4=-2c_2c_1+c_1^3=2c_2-1
		\end{align*} 
		where $c_1$ and $c_2$ are the Chern classes of $\mathcal{E}$. 
		We can then compute $T^4$:
		\begin{align*}
			T^4&= (2d\xi+(5d+1)H)^4\\&= 16d^4\xi^4+32d^3(5d+1)\xi^3H+24d^2(5d+1)^2\xi^2H^2+8d(5d+1)^3\xi\cdot H^3+ (5d+1)^4H^4\\
			&= 16d^4(2c_2-1)+32(5d^4+d^3)(-c_2+1)-24d^2(5d+1)^2+8d(5d+1)^3.
		\end{align*}
		
		Now, $T^4>0$ yields
		\begin{align*}
			-32(4 d^4+d^3)c_2> -544d^4-392d^3-96d^2-8d\\
		\end{align*}
		so that 
		\begin{align*}
			c_2<\frac{68d^3+49d^2+12d+1}{4(4d^3+d^2)}=\frac{(17d^2+8d+1)}{4d^2}.
			\\\
		\end{align*}
		It can be proved that $\phi(d):=\frac{(17d^2+8d+1)}{4d^2}$ is a strictly decreasing function if $d>0$. Since $d$ is a positive integer, we get $c_2<\phi(1)=\frac{13}{2}$. Since $c_2$ is also an integer, we get $c_2\le 6$. Observe that if $c_2=6$, then $d=1$, because $\phi(d)\le 5$ for all $d\ge 2$. 
		
		Recall the relation $c_1c_2\equiv 0$ mod $2$ for the Chern classes of a rank two vector bundle on $\mathbb{P}^3$ (see \cite[Corollary 2.2]{Har78}). In our setting, we deduce that $c_2$ must be even.  
	\end{proof}
	
	\begin{remark}\label{RemarkNef}
		Observe that $-K_{X}\equiv2\xi+(4-c_1(\mathcal{E}))H=2\xi_{\mathcal{E}(3)}-H$. 
		Recall that if $\mathbf{l}$ is an exceptional line in $X$, $-K_{X}\cdot \mathbf{l}=-1$, $H\cdot \mathbf{l}=\deg(\pi(\mathbf{l}))\ge 1$, and for all irreducible curves $C$ that are not exceptional lines we have $-K_{X}\cdot C\ge 1$, $H\cdot C\ge 0$. It follows that $\xi_{\mathcal{E}(3)}=\frac{1}{2}(-K_{X}+H)$ is  nef. Notice that it is ample if $d=\deg(\pi(\mathbf{l}_0))>1$.  
	\end{remark}
	
	We are now ready for the proof of the first part of the main theorem.
	
	\begin{proof}[Proof of Theorem \ref{TheoremNonFanoCase} (a)]
		First notice that in this setting $\mathcal{E}$ cannot be decomposable. Indeed, up to tensoring with a line bundle, $\mathcal{E}$ would be isomorphic to $\mathcal{O}\oplus\mathcal{O}(q)$ for some $q\in\mathbb{N}$. If $q=0$, then $X\simeq \mathbb{P}^3\times\mathbb{P}^1$ would be already Fano. If $q>0$, then the second extremal contraction of $X$ would be divisorial (its exceptional locus would coincide with the section associated with the surjection $\mathcal{O}\oplus\mathcal{O}(q)\twoheadrightarrow\mathcal{O}$). This would imply that $X$ has no small extremal contractions, against our assumptions.  
		
		Moreover, recall that in this setting $\mathcal{E}(3)$ must be nef by Remark \ref{RemarkNef} and $c_2$ is an even integer at most equal to $6$, by Lemma \ref{LemmaBoundc2}.
		
		Let $a\ge 0$ be the maximal integer with $H^0(\mathcal{E}(-a))\ne 0$. A nonzero section $s\in H^0(\mathcal{E}(-a))$ then has no divisorial zeros. If $Z:=\{s=0\}\ne \emptyset$, then $Z$ is a locally complete intersection curve of degree
		\begin{align}\label{Deg(Z)}
			\deg(Z)=c_2(\mathcal{E}(-a))=c_2(\mathcal{E})+a(a+1).
		\end{align}
		Observe that $\deg(Z)$ is even, because $c_2(\mathcal{E})$ is even. 
		
		Assume $a\ge1$. We claim that $s$ does not vanish anywhere. Indeed, if $Z\ne\emptyset$, for a line $L$ meeting $Z$ in a finite number of points, we have 
		\begin{align*}
			\mathcal{E}(-a)|_L=\mathcal{O}_L(m)\oplus\mathcal{O}_L(n)
		\end{align*} 
		with $m\ge 1$, $m+n=\det(\mathcal{E}(-a)|_L)=-1-2a$. Here, $m$ coincides with the number of points in $L\cap Z$, counted with multiplicities. Indeed, the intersection points coincide with the zero locus of $s|_L\in H^0(\mathcal{E}(-1)|_L)=H^0(\mathcal{O}_L(m))\oplus H^0(\mathcal{O}_L(n))=H^0(\mathcal{O}_L(m))$. 
		
		By the nefness of $\mathcal{E}(3)$, the only acceptable case is $a=1$ and $m=1,n=-4$. However, in this case the degree of $Z$ is $1$, which is a contradiction by (\ref{Deg(Z)}). Therefore $Z=\emptyset$, and $\mathcal{E}$ is decomposable \cite[Remark 1.0.1]{Har78}, another contradiction. 
		
		Thus suppose $a=0$, so that $H^0(\mathcal{E})\ne 0$. Again, consider the zero set $Z$ of a nonzero section of $\mathcal{E}$. Since $\mathcal{E}$ is not decomposable, $Z\ne\emptyset$. Observe that because of the parity of $c_2$, as before we see that $Z$ cannot be a single line. Consider a line $L$ that meets $Z$ in at least two points. Then
		\begin{align*}
			\mathcal{E}|_L=\mathcal{O}_L(m)\oplus\mathcal{O}_L(n)
		\end{align*}
		with $m\ge 2$, $m+n=-1$. Therefore since $\mathcal{E}(3)$ is nef, the only acceptable case is $m=2$, $n=-3$. Observe that, by Lemma \ref{LemmaRmk2.5}, the negative section of $\pi^{-1}(L)\simeq \mathbb{F}_5$ would have anticanonical degree $-1$. Since there are infinitely many such lines $L$, we would have infinitely many irreducible curves with negative anticanonical degree in $X$, a contradiction by Lemma \ref{LemmaSQMFano}. 
		
		We conclude that $H^0(\mathcal{E})=0$, i.e. $\mathcal{E}$ is stable. The Schwarzenberger inequality $c_1^2\le 4c_2$ (see \cite[Lemma 3.2]{Har78}) and the parity of $c_2$ imply that $c_2\ge 2$. Since $c_2\le 6$ by Lemma \ref{LemmaBoundc2}, we get $c_2\in\{2,4,6\}$.    
	\end{proof}
	\section{The case $c_2(\mathcal{E})=2$}
	In this section we are going to prove Theorem \ref{TheoremNonFanoCase} (b) and Theorem \ref{TheoremBir}, by studying the birational geometry of $X=\mathbb{P}(\mathcal{E})$, where $\mathcal{E}$ is a rank two stable vector bundle on $\mathbb{P}^3$ with $c_1=-1$ and $c_2=2$.  We denote by $\xi$ the tautological divisor of $X$, and by $H$ the pullback in $X$ of a plane in $\mathbb{P}^3$. The moduli space of these vector bundles is fully described in \cite{HS81}. Below, we summarize the results from \cite{HS81} that we are going to use.
	\begin{proposition}\label{PropositionZ}
		The vector space $H^0(\mathcal{E}(1))$ is one-dimensional. Let $Z$ be the zero set of a nonzero section $s\in H^0(\mathcal{E}(1))$. Then, we can choose homogeneous coordinates $[x:y:z:w]$ for $\mathbb{P}^3$ such that the ideal defining $Z$ is of the form $\mathcal{I}_Z=(x^2,xy,y^2,xf-yg)$, where $f,g$ are two homogeneous quadratic forms in the coordinates $z,w$ that do not have common zeros in $\mathbb{P}^1_{[w:z]}$. 
	\end{proposition}
	\begin{proof}
		This follows from Propositions 1.1, 2.1 and 2.2 in \cite{HS81}. 
	\end{proof}
	
	\begin{proposition}[\cite{HS81}, Proposition 2.3]\label{PropSectionE(2)}
		The vector bundle $\mathcal{E}(2)$ has a section whose zero set is a disjoint union of two irreducible conics $\mathcal{C}_1$ and $\mathcal{C}_2$. In particular, there is an exact sequence
		\begin{align*}
			0\to\mathcal{O}\longrightarrow\mathcal{E}(2)\longrightarrow \mathcal{I}_{\mathcal{C}}(3)\to 0
		\end{align*}
		where $\mathcal{C}:=\mathcal{C}_1\cup\mathcal{C}_2$. The planes containing the conics meet along the line $L_0=\{x=y=0\}$, the support of the ideal $\mathcal{I}_Z$. 
	\end{proposition}

	\begin{proposition}\label{PropositionJumping}
		Let $L_0$ be the reduced support of the ideal $\mathcal{I}_Z$ in Proposition \ref{PropositionZ}, which coincides with the intersection of the planes containing the conics $\mathcal{C}_1$ and $\mathcal{C}_2$ in Proposition \ref{PropSectionE(2)}. We have the following:
		\begin{itemize}
			\item[(i)] $\mathcal{E}|_{L_0}\simeq \mathcal{O}_{L_0}(2)\oplus\mathcal{O}_{L_0}(-3)$.
			\item[(ii)] There is a two dimensional locally closed subvariety of lines $L$ in the Grassmann variety $Gr(1,3)$ such that $\mathcal{E}|_L\simeq \mathcal{O}_L(1)\oplus\mathcal{O}_L(-2)$. These lines correspond to the ones that meet $L_0$ in one point and whose scheme-theoretical intersection with $Z$ has length two, i.e. they are \emph{tangent} to $Z$ at that point.
			\item[(iii)] If $L$ is any other line, $\mathcal{E}|_{L}\simeq \mathcal{O}\oplus\mathcal{O}(-1)$. 
		\end{itemize}
	\end{proposition}
	\begin{proof}
		It follows from \cite[Proposition 4.1]{HS81}.
	\end{proof}
	
	\begin{remark}
		We use the same choice of coordinates as in
		Proposition \ref{PropositionZ}. Let
		$p=[0:0:z:w]\in L_0=\{x=y=0\}$.
		The Zariski tangent space $T_pZ$ is defined by
		\[
		f(p)\,dx-g(p)\,dy=0
		\]
		in $T_p\mathbb P^3$. In particular, the unique plane
		through $L_0$ whose tangent space at $p$ equals $T_pZ$ is
		\[
		\{f(p)x-g(p)y=0\}.
		\]
		
		Since $f|_{L_0}$ and $g|_{L_0}$ have no common zero,
		they define a morphism
		\[
		\tau\colon L_0\simeq\mathbb P^1\longrightarrow\mathbb P^1,
		\qquad
		p\longmapsto[f(p):-g(p)].
		\]
		Since $f|_{L_0}$ and $g|_{L_0}$ are linearly independent and generate $\mathcal{O}_{\mathbb{P}^1}(2)$ without base points,
		then $\tau$ has degree two and it is surjective.
		Consequently, every plane containing $L_0$ is tangent to
		$Z$ at two points of $L_0$, counted with multiplicity. 
	\end{remark}
	
	\begin{remark}\label{RemarkTangentLines}
		We now describe the lines in $\mathbb P^3$ that are
		tangent to $Z$ at a point of $L_0$. More precisely, we
		say that a line $L\ne L_0$ is tangent to $Z$ at
		$p=L\cap L_0$ if its tangent direction at $p$ is
		contained in $T_pZ$.
		
		Every plane containing
		$L_0=\{x=y=0\}$ is of the form
		\[
		\mathcal P_{\alpha,\beta}
		=\{\alpha x+\beta y=0\},
		\qquad [\alpha:\beta]\in\mathbb P^1.
		\]
		Set
		\[
		h_{\alpha,\beta}
		=(\beta f+\alpha g)|_{L_0}.
		\]
		By the preceding remark, $\mathcal P_{\alpha,\beta}$
		is tangent to $Z$ at precisely the points of
		$V(h_{\alpha,\beta})\subset L_0$, counted with
		multiplicity.
		
		We first examine the scheme-theoretic intersection of
		$Z$ with $\mathcal P_{\alpha,\beta}$. If $\alpha\ne0$,
		identifying $\mathcal P_{\alpha,\beta}$ with
		$\mathbb P^2_{[y:z:w]}$ by setting
		$x=-\frac{\beta}{\alpha}y$, we obtain
		\[
		\mathcal{I}_Z\mathcal O_{\mathcal P_{\alpha,\beta}}
		=(y^2,y(\beta f+\alpha g)).
		\]
		Its support is $L_0$, and its embedded point is supported
		at each zero of $h_{\alpha,\beta}$. Thus, if
		$h_{\alpha,\beta}$ has two distinct roots, there are two
		distinct embedded points; if it is a perfect square,
		there is one embedded point supported at a single point
		of $L_0$, with multiplicity two. The same conclusion
		holds for $\alpha=0$, since in this case
		\[
		\mathcal{I}_Z\mathcal O_{\mathcal P_{0,1}}=(x^2,xf),
		\]
		and $h_{0,1}=f|_{L_0}$.
		
		We can now count the planes for which
		$h_{\alpha,\beta}$ is a perfect square. The perfect
		squares form a conic $\kappa$ in
		\[
		\mathbb P\bigl(H^0(L_0,\mathcal O_{L_0}(2))\bigr)
		\simeq\mathbb P^2.
		\]
		Since $f|_{L_0}$ and $g|_{L_0}$ are linearly independent,
		the forms $h_{\alpha,\beta}$ parametrize a projective
		line in this plane. The conic $\kappa$ contains no projective
		line, so the pencil meets it in two points counted
		with multiplicity. Hence there are either one or two
		distinct planes for which $h_{\alpha,\beta}$ is a
		perfect square. In fact, a tangent line to the conic $\kappa=\ell^2$ in $\mathbb P\bigl(H^0(L_0,\mathcal O_{L_0}(2))\bigr)$ consists of forms divisible by $\ell$, contradicting the absence of common zeros of $f$ and $g$. Therefore $h_{\alpha,\beta}$ is a perfect square for exactly two distinct choices of the plane $\mathcal{P}_{\alpha,\beta}$.  
		
		Finally, let $L\ne L_0$ be a line contained in
		$\mathcal P_{\alpha,\beta}$, and set $p=L\cap L_0$.
		The tangent direction of $L$ at $p$ belongs to
		$T_p\mathcal P_{\alpha,\beta}$, defined by
		\[
		\alpha\,dx+\beta\,dy=0.
		\]
		By the description of $T_pZ$ in the preceding remark,
		this direction belongs to $T_pZ$ if and only if
		\[
		\beta f(p)+\alpha g(p)=0,
		\]
		or equivalently, $p\in V(h_{\alpha,\beta})$.
		Therefore,
		\[
		\boxed{
			\begin{gathered}
				L\subset\mathcal P_{\alpha,\beta}
				\text{ is tangent to }Z\text{ at }L\cap L_0\\
				\Longleftrightarrow\
				L\cap L_0\in V(h_{\alpha,\beta}).
		\end{gathered}}
		\]
	\end{remark}

	The results contained in \cite{HS81} are the starting point for our analysis. We will complete the study of the birational geometry of $X$ in three steps.
	\begin{enumerate}
		\item First, we are going to verify that $X$ satisfies some necessary conditions for being a SQM of a Fano 4-fold appearing in Lemma \ref{LemmaSQMFano}, i.e. that there exist finitely many pairwise disjoint \emph{exceptional lines} $\ell_1,\dots,\ell_r$ in $X$, and for every other irreducible curve $C\subset X$, we have $-K_X\cdot C\ge 1$.   
		
		In order to do so, we are going to fully describe the cones $\Nef(X)$, $\NE(X)$, and we are going to study in depth the geometry of the effective divisors in $|\xi_{\mathcal{E}(2)}|$ described in Lemma \ref{LemmaE(2)}. We are also going to use Proposition \ref{PropositionJumping} to identify one curve $\mathbf{l}_0$ in $X$ that has all properties that are necessary for being the only \emph{exceptional line} in $X$, i.e. it is a smooth rational curve with $\mathcal{N}_{\mathbf{l}_0/X}\simeq\mathcal{O}_{\mathbf{l}_0}(-1)^{\oplus 3}$.
		\item Secondly, we are going to study the geometry of the blow-up $\hat{X}$ of $X$ along the candidate exceptional line $\mathbf{l}_0$. 
		We are then going to show that $\hat{X}$ has another extremal contraction onto a smooth Fano variety $X^+$ that realizes $X^+$ as the flip of $X$ with $\mathbf{l}_0$ as its indeterminacy locus in $X$. 
		
		\item Thirdly, we are going to identify the fibers of the second extremal contraction of $X^+$, meaning the one that is not related to the flip $X\dasharrow X^+$. We are then going to show that such contraction is divisorial, and describe the full cone of contractions of $X^+$.      
	\end{enumerate}

	\subsection{The geometry of $X=\mathbb{P}(\mathcal{E})$.}
	
	Consider $X=\mathbb{P}(\mathcal{E})$ as above and its projection $\pi\colon X\to \mathbb{P}^3$. 	Denote by $\xi$ the tautological divisor of $X$, and by $H$ the pullback in $X$ of a plane in $\mathbb{P}^3$.  In this section, we describe the geometry of $X$. 
	\begin{tcolorbox}[
		title={Roadmap},
		colback=gray!5,
		colframe=black!60,
		boxrule=0.5pt,
		arc=1mm,
		left=2mm,
		right=2mm,
		top=1mm,
		bottom=1mm
		]
		\begin{itemize}
			\item We consider $L_0$, the jumping line for $\mathcal{E}$ appearing in Proposition \ref{PropositionJumping}, and the negative section $\mathbf{l}_0$ of its preimage $\pi^{-1}(L_0)\simeq\mathbb{F}_5$ in $X$. We prove that 
			\begin{align*}
				\NE(X)=\mathbb{R}_{\ge 0}[\mathbf{f}]+\mathbb{R}_{\ge 0}[\mathbf{l}_0] \qquad \Nef(X)=\mathbb{R}_{\ge 0}[H]+\mathbb{R}_{\ge 0}[\xi+3H]
			\end{align*} 
			where $[\mathbf{f}]$ is the numerical class of a one-dimensional fiber of $\pi\colon X\to \mathbb{P}^3$.  
			\item We describe the unique member $F$ of the linear system $|\xi+H|$ and the general member $D$ of the linear system $|\xi+2H|$. 
			\item We observe that $\mathbf{l}_0$ is contained in both $F$ and $D$, and we study the birational geometry of $D$ in detail. We prove that
			
			\[-K_D\cdot \mathbf{l}_0=0\qquad\text{and}\qquad-K_D\cdot C\ge 1\qquad \text{for all integral curves }C\subset D,\;C\ne\mathbf{l}_0.\]
			This is the key to prove that 
			\[-K_X\cdot \mathbf{l}_0=-1\qquad\text{and}\qquad-K_X\cdot C\ge 1 \qquad\text{for all integral curves }C\subset X,\;C\ne\mathbf{l}_0.\]
			In fact, we prove that $X$ satisfies necessary conditions for being a SQM of a Fano 4-fold (in the sense of Lemma \ref{LemmaSQMFano}), with $\mathbf{l}_0$ being its only \emph{exceptional line} (see Proposition \ref{PropositionNegCurves}).     
		\end{itemize}
	\end{tcolorbox}
	\paragraph{The Mori cone of $X$.}
	Consider $S_0:=\pi^{-1}(L_0)$. By Proposition \ref{PropositionJumping}, $S_0= \mathbb{P}(\mathcal{O}_{L_0}(2)\oplus\mathcal{O}_{L_0}(-3))\simeq\mathbb{F}_5$. Let $\mathbf{l}_0$ be the negative section of $S_0$. Then by Lemma \ref{LemmaRmk2.5}, $-K_{X}\cdot \mathbf{l}_0=-K_{\mathbb{P}^3}\cdot L_0-5=-1$. 
	Notice that $-K_{X}\equiv2\xi+5H$ in $N^1(X)$, and $H\cdot \mathbf{l}_0=1$ by the projection formula. This yields $\xi\cdot \mathbf{l}_0=-3$, so that in $X$ we have the following intersection table:
	\begin{center}
		\begin{tabular}{|c|c|c|}
			\hline
			$\cdot$ & $\xi$ & $H$ \\
			\hline
			$\mathbf{l}_0$ & $-3$ & $1$\\
			\hline
			$\mathbf{f}$ & $1$ & $0$\\
			\hline
		\end{tabular}
	\end{center}
	where $\mathbf{f}$ is a fiber of $\pi$. 
	
	\begin{lemma}\label{LemmaTnef}
		The vector bundle $\mathcal{E}(3)$ is globally generated. In particular, $\xi_{\mathcal{E}(3)}=\xi+3H$ is globally generated.
	\end{lemma}
	\begin{proof}
		By \cite[Proposition 2.2]{HS81}, we have $H^1(\mathcal{E}(2))=H^2(\mathcal{E}(1))=H^3(\mathcal{E})=0$. This is enough to use Castelnuovo's Theorem \cite[p. 99]{Mum66} and conclude that $\mathcal{E}(k)$ is globally generated for all $k\ge 3$. 
	\end{proof}
	\begin{lemma}\label{LemmaNE}
		Denote $T:=\xi+3H$. We have the following:
		\begin{itemize}
			\item $\Nef(X)=\mathbb{R}_{\ge 0}[H]+\mathbb{R}_{\ge 0}[T]$.
			\item $\NE(X)=\mathbb{R}_{\ge 0}[\mathbf{f}]+\mathbb{R}_{\ge 0}[\mathbf{l}_0]$.
		\end{itemize}
	\end{lemma}
	\begin{proof}
		Notice that $T$ is nef by Lemma \ref{LemmaTnef}, and $T\cdot\mathbf{l}_0=0$. Consider $C$ any irreducible curve in $X$. Then there exist $\alpha,\beta\in\mathbb{R}$ such that $C\equiv\alpha \mathbf{l}_0+\beta \mathbf{f}$. Intersect with the nef divisors $H$ and $T$, and impose non-negativity:
		\begin{align*}
			0\le C\cdot H= \alpha\\
			0\le C\cdot T=\beta.
		\end{align*} 
		Therefore $\NE(X)=\mathbb{R}_{\ge 0}[\mathbf{f}]+\mathbb{R}_{\ge 0}[\mathbf{l}_0]$.
		
		Dually, one can see that $\Nef(X)=\mathbb{R}_{\ge 0}[H]+\mathbb{R}_{\ge 0}[T]$.  
	\end{proof}
	
	\paragraph{The linear systems $|\xi+H|$ and $|\xi+2H|$ in $X$.}
	Based on the information contained in \cite{HS81}, we now describe the members of the following linear systems: $|\xi+H|=|\xi_{\mathcal{E}(1)}|$, $|\xi+2H|=|\xi_{\mathcal{E}(2)}|$. The members of these linear systems will serve as important starting points for our exploration of the geometry of $X$. Notice that $(\xi+H)\cdot\mathbf{l}_0=-2$ and $(\xi+2H)\cdot \mathbf{l}_0=-1$, thus $\mathbf{l}_0$ is contained in every member of both linear systems.  
	
	\begin{lemma}\label{LemmaE(1)}
		There is only one divisor $F$ in the linear system $|\xi+H|$, and it is realized through $\pi|_F\colon F\to \mathbb{P}^3$ as the blow-up of $\mathbb{P}^3$ along the ideal $\mathcal{I}_Z$ in Proposition \ref{PropositionZ}. 
	\end{lemma}
	\begin{proof}
		This follows directly from Proposition \ref{PropositionZ}, from the well known correspondence $H^0(X,\xi_{\mathcal{E}(1)})\simeq H^0(\mathbb{P}^3,\mathcal{E}(1))$ and from the fact that a regular section of a rank two vector bundle with codimension-two zero scheme gives a projectivized rank-one quotient which is identified with the blow-up of its ideal.   
	\end{proof}
	
	\begin{lemma}\label{LemmaE(2)}
		The linear system $|\xi+2H|$ is six-dimensional. A general member $D$ is isomorphic, via $\pi|_D\colon D\to\mathbb P^3$, to the blow-up of $\mathbb P^3$ along two disjoint irreducible conics $\mathcal C_1$ and $\mathcal C_2$.
		
		Moreover, for every unordered pair of distinct points
		\[
		[\alpha_1:\beta_1],[\alpha_2:\beta_2]\in\mathbb P^1,
		\]
		there exists a divisor
		\[
		D_{\alpha_1,\beta_1}^{\alpha_2,\beta_2}\in|\xi+2H|
		\]
		such that $\pi|_{D_{\alpha_1,\beta_1}^{\alpha_2,\beta_2}}$ realizes $D_{\alpha_1,\beta_1}^{\alpha_2,\beta_2}$ as the blow-up of $\mathbb P^3$ along two disjoint irreducible conics
		\[
		\mathcal C_{\alpha_1,\beta_1}\subset\mathcal P_{\alpha_1,\beta_1},
		\qquad
		\mathcal C_{\alpha_2,\beta_2}\subset\mathcal P_{\alpha_2,\beta_2},
		\]
		where
		\[
		\mathcal P_{\alpha_i,\beta_i}
		:=
		\{\alpha_i x+\beta_i y=0\}
		\]
		are planes containing $L_0=\{x=y=0\}$, with the same choice of coordinates as in Proposition~\ref{PropositionZ}.		
	\end{lemma}
	
	\begin{proof}
		By \cite[Proposition 2.2]{HS81}, we have
		\(
		\dim H^0(\mathcal E(2))=7.
		\)
		Using the natural identification
		\(
		H^0(X,\mathcal O_X(\xi+2H))
		\simeq
		H^0(\mathbb P^3,\mathcal E(2)),
		\)
		we conclude that $|\xi+2H|$ is six-dimensional.

		In the proof of Proposition~2.3 of \cite{HS81}, the authors prove that there exists a section
		\[
		s\in H^0(\mathbb P^3,\mathcal E(2))
		\]
		whose zero locus is the disjoint union of two irreducible conics
		$\mathcal C_1$ and $\mathcal C_2$, with
		\[
		\mathcal C_1\subset\{x=0\},
		\qquad
		\mathcal C_2\subset\{y=0\}.
		\]
		Via the above identification, $s$ determines a divisor
		$D\in|\xi+2H|$. Since the zero locus of $s$ is
		$\mathcal C_1\sqcup\mathcal C_2$, the morphism
		\[
		\pi|_D\colon D\longrightarrow\mathbb P^3
		\]
		realizes $D$ as the blow-up of $\mathbb P^3$ along
		$\mathcal C_1\sqcup\mathcal C_2$.
		
		We now show that the same construction can be carried out
		with any pair of distinct planes containing $L_0$. More
		precisely, the proof of \cite[Proposition~2.3]{HS81} gives,
		for every pair of distinct points
		\[
		[\alpha_1:\beta_1],[\alpha_2:\beta_2]\in\mathbb P^1,
		\]
		a linear subspace
		\[
		W_{\alpha_1,\beta_1}^{\alpha_2,\beta_2}
		\subset H^0(\mathbb P^3,\mathcal E(2))
		\]
		of dimension at least $5$, whose general section vanishes
		along the disjoint union of two irreducible conics
		\[
		\mathcal C_{\alpha_1,\beta_1}
		\subset\mathcal P_{\alpha_1,\beta_1},
		\qquad
		\mathcal C_{\alpha_2,\beta_2}
		\subset\mathcal P_{\alpha_2,\beta_2}.
		\]
		
		Set
		\[
		Q:=\left((\mathbb P^1\times\mathbb P^1)\setminus\Delta\right)
		/\mathfrak S_2,
		\]
		where $\Delta$ is the diagonal and $\mathfrak S_2$ acts by
		interchanging the two factors. Since
		\[
		\operatorname{Sym}^2\mathbb P^1
		\simeq\mathbb P^2,
		\]
		we have $Q\simeq\mathbb P^2\setminus C$, where
		$C\subset\mathbb P^2$ is a smooth conic. In particular,
		$\dim Q=2$.
		
		The construction in the proof of
		\cite[Proposition~2.3]{HS81} expresses the spaces
		$W_{\alpha_1,\beta_1}^{\alpha_2,\beta_2}$ in terms of
		linear conditions depending algebraically on the prescribed
		planes. Consequently, these spaces form an algebraic family
		over $Q$.
		
		For each $q\in Q$, let $U_q\subset W_q$ be the subset of
		sections whose zero loci are disjoint unions of two
		irreducible conics contained in the two planes parametrized
		by $q$. By the construction in
		\cite[Proposition~2.3]{HS81}, $U_q$ is dense in $W_q$.
		Moreover, the description of the zero loci of sections in $U_q$  implies
		that the subsets $U_q$ are pairwise disjoint.
		
		The resulting incidence family
		\[
		\mathcal U
		:=\bigcup_{q\in Q}\{q\}\times U_q
		\subset Q\times H^0(\mathbb P^3,\mathcal E(2))
		\]
		has dimension at least
		\[
		\dim\mathcal U\geq\dim Q+5=7.
		\]
		Since the projection
		\[
		\mathcal U\longrightarrow H^0(\mathbb P^3,\mathcal E(2))
		\]
		has finite fibers, by the uniqueness of the corresponding
		pair of planes, its image has dimension at least $7$.
		On the other hand,
		\[
		\dim H^0(\mathbb P^3,\mathcal E(2))=7.
		\]
		It follows that the image of $\mathcal U$ is dense in
		$H^0(\mathbb P^3,\mathcal E(2))$. Hence a general section
		of $\mathcal E(2)$ vanishes along the disjoint union of
		two irreducible conics contained in two distinct planes
		through $L_0$. The corresponding general member of
		$|\xi+2H|$ is therefore the blow-up of $\mathbb P^3$
		along these two conics.
	\end{proof}
	
	\begin{remark}\label{RemarkAlphaBeta}
		For every fixed pair of distinct points
		$[\alpha_1:\beta_1],[\alpha_2:\beta_2]\in\mathbb P^1$, the divisor
		$D_{\alpha_1,\beta_1}^{\alpha_2,\beta_2}$ in Lemma~\ref{LemmaE(2)} is not uniquely determined. Indeed, the proof of the lemma gives a linear subspace
		\(
		\mathbb P\!\left(
		W_{\alpha_1,\beta_1}^{\alpha_2,\beta_2}
		\right)
		\subset|\xi+2H|
		\)
		of dimension at least $4$, whose general member is the blow-up of $\mathbb P^3$ along a pair of disjoint irreducible conics contained in the fixed planes
		\[
		\mathcal P_{\alpha_1,\beta_1}
		=
		\{\alpha_1x+\beta_1y=0\},
		\qquad
		\mathcal P_{\alpha_2,\beta_2}
		=
		\{\alpha_2x+\beta_2y=0\}.
		\]

		The planes are therefore fixed, whereas the conics vary. In particular, the two conics are not uniquely determined by the pair of planes. Rather, the admissible conics form linear systems in the respective planes, and a general choice gives a pair of disjoint irreducible conics.
		
		Accordingly, whenever we use the notation
		$D_{\alpha_1,\beta_1}^{\alpha_2,\beta_2}$, we mean a divisor in
		$|\xi+2H|$ obtained from one of the infinitely many admissible pairs of disjoint conics contained in the fixed planes
		$\mathcal P_{\alpha_1,\beta_1}$ and $\mathcal P_{\alpha_2,\beta_2}$.
		
		When the discussion is independent of the choice of
		$[\alpha_i:\beta_i]$, we will omit these parameters and simply write
		$D$ for $D_{\alpha_1,\beta_1}^{\alpha_2,\beta_2}$ and
		$\mathcal P_i$ for
		$\mathcal P_{\alpha_i,\beta_i}$, $i=1,2$.
	\end{remark}

	\paragraph{The geometry of a general member $D$ of $|\xi+2H|$.}
	We now fix any divisor $D$ in $X$ belonging to the linear system $|\xi_{\mathcal{E}(2)}|=|\xi+2H|$, such that $D\simeq \Bl_{\mathcal{C}}\mathbb{P}^3$, and $\pi|_D\colon D\to \mathbb{P}^3$ coincides with the blowup of $\mathbb{P}^3$ along $\mathcal{C}$, where $\mathcal{C}$ is the disjoint union of two irreducible conics $\mathcal{C}_1$ and $\mathcal{C}_2$. The study of the birational geometry of $D$ will provide useful information also on the geometry of $X$, in particular it will allow us to prove that $\mathbf{l}_0$ is an \emph{exceptional line}, and $X$ satisfies the necessary conditions for being related to a Fano 4-fold by a SQM that appear in Lemma \ref{LemmaSQMFano}. 
	
	Let $P$ be the strict transform in $D$ of a general plane in $\mathbb{P}^3$, and denote by $E_i$ the exceptional divisor that maps onto $\mathcal{C}_i$, $i=1,2$. 
	
	Let $\mathit{r}$, $\mathit{r_0}$ be the transforms in $D$ of a general line $L$ and of the line $L_0$ in $\mathbb{P}^3$, respectively. Let $\mathit{e_i}$ be an exceptional fiber over a point of $\mathcal{C}_i$, $i=1,2$.
	In $D$ the following intersection table holds:
	\begin{center}
		\begin{tabular}{|c|c|c|c|}
			\hline
			$\cdot$ & $P$ & $E_1$ & $E_2$ \\
			\hline
			$\mathit{r}$ & $1$ & $0$ & $0$\\
			\hline
			$\mathit{e_1}$ & $0$ & $-1$ & $0$\\
			\hline
			$\mathit{e_2}$ & $0$ & $0$ & $-1$\\
			\hline
		\end{tabular}
	\end{center}
	
	Notice that, by definition, $L_0$ is the only line that meets each conic $\mathcal{C}_i$ in two points (counted with multiplicities). In particular we have $\mathit{r_0}\equiv\mathit{r}-2\mathit{e_1}-2\mathit{e_2}$ in $N_1(D)$.
	
	\begin{lemma}\label{LemmaD}
		The 3-fold $D$ is weak-Fano (i.e. $-K_D$ is big and nef), with Picard number $\rho_D=3$, and we have the following:
		\begin{itemize}
			\item $\Nef(D)=\mathbb{R}_{\ge 0}[P]+\mathbb{R}_{\ge 0}[2P-E_1]+\mathbb{R}_{\ge 0}[2P-E_2]$.
			\item $\NE(D)=\mathbb{R}_{\ge 0}[\mathit{r_0}]+\mathbb{R}_{\ge 0}[\mathit{e_1}]+\mathbb{R}_{\ge 0}[\mathit{e_2}]$. 
		\end{itemize}
		
		Moreover $\mathcal{N}_{D/X}\simeq \mathcal{O}_D(3P-E_1-E_2)$, the curve $\mathbf{l}_0$ is contained in $D$ and coincides with $\mathit{r}_0$. In particular $\mathbf{l}_0$ is the only effective member of its equivalence class in $\NE(X)$. 
		
		Finally, for all irreducible curves $C\ne \mathit{r_0}$ in $D$, we have:
		\begin{itemize}
			\item $-K_{X}\cdot C\ge 1$
			\item $-K_X\cdot C=1$ if and only if $C$ is the transform of a line contained in $\mathcal{P}_i$ that meets the conic $\mathcal{C}_j$ in one point, for $(i,j)\in\{(1,2),(2,1)\}$. 
		\end{itemize}  
	\end{lemma}
	\begin{proof}
		We have $\rho_D=3$. 
		Consider $Q_i=2P-E_i$ for $i=1,2$. Then $Q_i$ is represented by the strict transform in $D$ of a quadric in $\mathbb{P}^3$ containing the conic $\mathcal{C}_i$. The ideal $\mathcal{I}_{\mathcal{C}_i}$ of the plane conic is generated by its plane equation and its quadratic equation, so its twist by 2 is globally generated. Since $(\pi|_D)^{*}\mathcal{I}_{\mathcal{C}_i}(2)=\mathcal{O}_D(2P-E_i)$, we conclude that $Q_i$ is globally generated. The divisors $P$, $Q_1$, $Q_2$ are nef, and their intersections with $\mathit{r}_0$, $\mathit{e}_1$ and $\mathit{e}_2$ form the identity matrix, in these orders. Since the three curves are effective, duality gives exactly
		\[
		 \Nef(D)=\mathbb{R}_{\ge 0}[P]+\mathbb{R}_{\ge 0}[2P-E_1]+\mathbb{R}_{\ge 0}[2P-E_2], \qquad
		\NE(D)=\mathbb{R}_{\ge 0}[\mathit{r_0}]+\mathbb{R}_{\ge 0}[\mathit{e_1}]+\mathbb{R}_{\ge 0}[\mathit{e_2}].
		\]

		Observe that $-K_D\equiv Q_1+Q_2$ is nef. We now prove that it is also big.
		One can check that the following holds:
		\begin{align*}
			P^3=1;\quad E_1^3=E_2^3=-6;\quad P^2\cdot E_1=P^2\cdot E_2=0;\quad P\cdot E_1^2=P\cdot E_2^2=-2
		\end{align*}
		and all products involving both $E_1$ and $E_2$ are null. 
		Thus one can compute $(-K_D)^3=28>0$, which implies that $-K_D$ is big. 
		
		Let us now compute the normal bundle $\mathcal{N}_{D/X}$. Since $-K_X\equiv 2\xi+5H=2D+H$, adjunction gives 
		\[
		-K_D=(-K_X-D)|_D=D|_D+P.
		\]
		Using $-K_D=4P-E_1-E_2$, one recovers 
		\[
		\mathcal{N}_{D/X}=\mathcal{O}_D(D)\simeq\mathcal{O}_D(3P-E_1-E_2),\qquad -K_X|_D=7P-2E_1-2E_2.
		\]
		
		Now, observe that $D\cdot \mathbf{l}_0=(\xi+2H)\cdot \mathbf{l}_0=-1$, hence $\mathbf{l}_0$ is contained in $D$. Away from $\mathcal{C}_1\cup\mathcal{C}_2$, the map $\pi|_D$ is an isomorphism, and $\mathbf{l}_0$ maps to $L_0$. Taking closures shows that $\mathbf{l}_0$ coincides with the strict transform $\mathit{r}_0$. 
		
		Finally, we check that the anticanonical degree in $X$ of curves contained in $D$ is correctly described in the statement. 
		If $C\equiv\alpha\mathit{r}+\beta\mathit{e_1}+\gamma\mathit{e_2}$ is any curve in $D$, we have $-K_{X}\cdot C\equiv7\alpha+2\beta+2\gamma$. Suppose $C$ is irreducible. If $\alpha=0$, then $C\equiv\mathit{e_1}$ or $C\equiv\mathit{e_2}$, and $-K_{X}\cdot C=2$. On the other hand, if $\alpha\ge 1$, then $\pi(C)$ is an irreducible curve in $\mathbb{P}^3$. 
		
		Suppose $C\subset E_1$. Observe that $E_1\simeq\mathbb{P}_{\mathcal{C}_1}(\mathcal{N}_{\mathcal{C}_1/\mathbb{P}^3}^*)=\mathbb{P}_{\mathcal{C}_1}(\mathcal{O}_{\mathcal{C}_1}(-1)\oplus\mathcal{O}_{\mathcal{C}_1}(-2))\simeq \mathbb{P}_{\mathbb{P}^1}(\mathcal{O}_{\mathbb{P}^1}(-2)\oplus\mathcal{O}_{\mathbb{P}^1}(-4))\simeq \mathbb{F}_2$. Let $\Gamma^-$ be the negative section of $\mathbb{F}_2$, and $e_1$ be a fiber of the $\mathbb{P}^1$-bundle $E_1\to \mathcal{C}_1$. Then $\mathcal{O}_{E_1}(1)=[\Gamma^-]-2[e_1]$, and $E_1|_{E_1}=\mathcal{O}_{E_1}(-1)=-[\Gamma^-]+2[e_1]$. Moreover  $\NE(E_1)=\mathbb{R}_{\ge0}[\Gamma^-]+\mathbb{R}_{\ge0}[e_1]$. By intersecting $\Gamma^-$ with $P$, $E_1$, $E_2$, we get $\Gamma^-\equiv 2\mathit{r}-4\mathit{e_1}\equiv 2\mathit{r_0}+4\mathit{e_2}$ in $N_1(D)$. Therefore if $C\equiv\delta\Gamma^{-}+\epsilon e_1$ is in $\NE(E_1)$, then $\delta,\epsilon\in\mathbb{Z}_{\ge 0}$, and $C\equiv2\delta\mathit{r}+(-4\delta+\epsilon)\mathit{e_1}\equiv 2\delta\mathit{r_0}+\epsilon\mathit{e_1}+4\delta\mathit{e_2}$ in $\NE(D)$.
		
		We can then conclude that if $C\subset E_1$, then $C\equiv2\delta\mathit{r}+(\epsilon-4\delta)\mathit{e_1}$ for some $\delta,\epsilon\in\mathbb{Z}_{\ge 0}$ and $-K_{X}\cdot C=14\delta+2(\epsilon-4\delta)=6\delta+2\epsilon\ge 2$. We can argue analogously if $C\subset E_2$. 
		
		On the other hand, suppose $C\equiv\alpha\mathit{r}+\beta\mathit{e_1}+\gamma\mathit{e_2}$ is not contained in $E_1$, nor in $E_2$. Consider the effective strict plane transforms 
		\[
		\Lambda_i:=P-E_i.
		\]
		If $\pi(C)\nsubseteq \mathcal{P}_i$, then $C\nsubseteq\Lambda_i$, hence 
		\[
		0\le \Lambda_i\cdot C=\alpha+\beta.
		\] 
		If $\pi(C)\subset \mathcal{P}_i$, then the ideal of $\mathcal{C}_i$ on the normalization of $\pi(C)$ is generated by the conic equation, whose zero divisor has degree $2\alpha$; hence $-\beta=2\alpha$. 
		
		If $\pi(C)$ is contained in neither $\mathcal{P}_1$ nor $\mathcal{P}_2$, then we obtain
		\[
		-K_X\cdot C=7\alpha+2\beta+2\gamma\ge 3\alpha\ge 3. 
		\]
		If $\pi(C)\subset \mathcal{P}_1$, and $\pi(C)\ne L_0$, then $\pi(C)\nsubseteq \mathcal{P}_2$ and 
		\[
		-K_X\cdot C=3\alpha+2\gamma\ge \alpha\ge 1.
		\]
		Equality with $1$ forces $\alpha=1$ and $\gamma=-1$. Thus $\pi(C)$ is a line in $\mathcal{P}_1$ meeting $\mathcal{C}_2$ with length one. Conversely, any such line has anticanonical degree one after strict transform. The case with indices exchanged is identical. If $\pi(C)=L_0$, then $C=\mathit{r}_0$ and $\beta=\gamma=-2$, so $-K_X\cdot \mathit{r}_0=-1$. 
	\end{proof}
	
	\begin{corollary}\label{CorollaryDeg1}
		For $i=1,2$, the strict transform $\Lambda_i$ in $D$ of the plane $\mathcal{P}_i\subset \mathbb{P}^3$ contains an open dense subset $U_i=\Lambda_i\setminus \pi^{-1}(L_0)$ such that for every point $p\in U_i$, there exists a curve in $D$ that contains $p$, has anticanonical degree one in $X$ and whose image in $\mathbb{P}^3$ is a line. 
	\end{corollary}
	\begin{proof}
		Fix $(i,j)\in\{(1,2),(2,1)\}$. Consider all lines in $\mathcal{P}_i$ passing through one of the (at most) two points in $\mathcal{C}_j\cap \mathcal{P}_i$, with the exception of $L_0$. For every point $q$ in the open subset $V_i:=\mathcal{P}_i\setminus\{L_0\}$, there exists one of such lines passing through $q$.  It follows that every point $p$ in the open subset $U_i:=\Lambda_i\setminus\pi^{-1}(L_0)\subseteq \Lambda_i$ is contained in the transform of one of those lines in $D$.  By Lemma \ref{LemmaD}, such transforms have anticanonical degree $1$.   
	\end{proof}
	
	\paragraph{The exceptional line $\mathbf{l}_0$.} 
	We can now establish the key property of $\mathbf{l}_0$: it is the unique curve in $X$ with negative anticanonical degree, and it has precisely the normal bundle of an exceptional line.
	
	\begin{proposition}\label{PropositionNegCurves}
		In $X$, we have 
		\begin{itemize}
			\item $-K_X\cdot \mathbf{l}_0=-1$;
			\item $-K_X\cdot C\ge 1$ for all integral curves $C\ne\mathbf{l}_0$ in $X$.
		\end{itemize}
		Moreover  $\mathcal{N}_{\mathbf{l}_0/X}\simeq\mathcal{O}_{\mathbf{l}_0}(-1)^{\oplus 3}$, hence $\mathbf{l}_0$ is an \emph{exceptional line}.
	\end{proposition}
	\begin{proof}
		Note that $-K_{X}\equiv2\xi+5H=2D+H$. We already computed $-K_X\cdot \mathbf{l}_0=-1$. 
		
		Let $C\ne\mathbf{l}_0$ be an integral curve in $X$. Suppose that $C$ is contained in some divisor $D\in|\xi+2H|$ as in Lemma \ref{LemmaE(2)}. Then by Lemma \ref{LemmaD}, we have $-K_X\cdot C\ge 1$. On the other hand, suppose $C$ is not contained in $D$, so that $D\cdot C\ge 0$. If $C$ is not contracted by $\pi$, then $H\cdot C\ge 1$ and $-K_X\cdot C=2D\cdot C+H\cdot C\ge 1$. Otherwise, if $C$ is contracted by $\pi$, then $-K_X\cdot C=2\ge 1$. 
		
		Finally, using the fact that $\mathit{r_0}$ in $D$ is the complete intersection of the transforms of the two planes containing the two conics $\mathcal{C}_1$, $\mathcal{C}_2$, and the fact that, by adjunction, $D\cdot \mathbf{l}_0=-1$, one can prove that $\mathcal{N}_{\mathbf{l}_0/X}\simeq \mathcal{O}_{\mathbf{l}_0}(-1)^{\oplus 3}$. Indeed, consider the short exact sequence 
		\begin{align}\label{SequenceNormal}
			0\to \mathcal{N}_{\mathbf{l}_0/D}\to\mathcal{N}_{\mathbf{l}_0/X}\to \mathcal{N}_{D/X}|_{\mathbf{l}_0}\to 0.
		\end{align}
		We have $\mathcal{N}_{\mathbf{l}_0/D}\simeq \mathcal{O}_{\mathbf{l}_0}(P-E_1)\oplus\mathcal{O}_{\mathbf{l}_0}(P-E_2)\simeq\mathcal{O}_{\mathbf{l}_0}(-1)\oplus\mathcal{O}_{\mathbf{l}_0}(-1)$, and $\mathcal{N}_{D/X}|_{\mathbf{l}_0}\simeq \mathcal{O}_{\mathbf{l}_0}(-1)$. Now, 
		\begin{align*}
			\Ext^1(\mathcal{O}_{\mathbb{P}^1}(-1),\mathcal{O}_{\mathbb{P}^1}(-1)\oplus\mathcal{O}_{\mathbb{P}^1}(-1))\simeq& \Ext^1(\mathcal{O}_{\mathbb{P}^1},(\mathcal{O}_{\mathbb{P}^1}(-1)\oplus\mathcal{O}_{\mathbb{P}^1}(-1))\otimes\mathcal{O}_{\mathbb{P}^1}(1))\\
			\simeq& \Ext^1(\mathcal{O}_{\mathbb{P}^1},\mathcal{O}_{\mathbb{P}^1}\oplus\mathcal{O}_{\mathbb{P}^1})\\
			\simeq& H^1(\mathbb{P}^1,\mathcal{O}_{\mathbb{P}^1}^{\oplus 2})=0
		\end{align*}
		Thus (\ref{SequenceNormal}) splits and $\mathcal{N}_{\mathbf{l}_0/X}\simeq \mathcal{O}_{\mathbf{l}_0}(-1)^{\oplus 3}$.
		
	\end{proof}

	\subsection{The existence of a flip to a Fano 4-fold $X^+$}\label{SectionX+}
	
	In this section we are going to prove that there exists a flip $X\dashrightarrow X^+$, where $X^+$ is a smooth Fano 4-fold, and we are going to describe the geometry of $X^+$.
	
	Consider the blow-up $\sigma\colon \hat{X}\to X$ of $X$ along $\mathbf{l}_0$, with exceptional divisor $E$. Let $C\ne \mathbf{l}_0$ be an integral curve in $X$. We define 
	\begin{align*}
		\mu(C):=E\cdot \hat{C}=-\deg(\mathcal{I}_{\mathbf{l}_0}\mathcal{O}_{\tilde{C}}),
	\end{align*} 
	where $\tilde{C}$ is the normalization of $C$ and the pull-back ideal is viewed as an invertible ideal on that smooth curve. Note that in general $\mu(C)\ge\length(C\cap \mathbf{l}_0)$, and if $C$ is smooth then $\mu(C)=\length(C\cap \mathbf{l}_0)$. We now state a technical lemma that will be useful in several occasions. 
	
	\begin{lemma}\label{LemmaLengthInclusions}
		Let $U$ be a scheme over a field $k$, and let $C_1,C_2,W$ be closed integral subschemes of $U$, with $C_1$ a projective curve. Suppose that $C_2\subseteq W$ and $C_1\not\subseteq W$. Then
		\[
		\deg(\mathcal{I}_{C_2}\mathcal{O}_{\widetilde{C}_1})
		\geq
		\deg(\mathcal{I}_W\mathcal{O}_{\widetilde{C}_1}),
		\]
		where $\widetilde{C}_1$ denotes the normalization of $C_1$.
	\end{lemma}
	
	\begin{proof}
		Since $C_2\subseteq W$, we have an inclusion of ideal sheaves
		\[
		\mathcal{I}_W\hookrightarrow\mathcal{I}_{C_2}.
		\]
		Let $\nu:\widetilde{C}_1\to C_1\hookrightarrow U$ be the normalization morphism. By extending the ideals to $\widetilde{C}_1$, we obtain an inclusion
		\[
		\mathcal{I}_W\mathcal{O}_{\widetilde{C}_1}
		\subseteq
		\mathcal{I}_{C_2}\mathcal{O}_{\widetilde{C}_1}.
		\]
		Since $C_1\not\subseteq W$, both ideals are nonzero. Moreover, as $\widetilde{C}_1$ is a smooth integral curve, both are invertible ideal sheaves. Thus we obtain an exact sequence
		\[
		0\longrightarrow
		\mathcal{I}_W\mathcal{O}_{\widetilde{C}_1}
		\longrightarrow
		\mathcal{I}_{C_2}\mathcal{O}_{\widetilde{C}_1}
		\longrightarrow \mathcal{Q}
		\longrightarrow 0,
		\]
		where $\mathcal{Q}$ is a torsion sheaf. Taking degrees, we conclude that
		\[
		\begin{aligned}
			\deg(\mathcal{I}_{C_2}\mathcal{O}_{\widetilde{C}_1})
			&=\deg(\mathcal{I}_W\mathcal{O}_{\widetilde{C}_1})
			+\deg(\mathcal{Q})\\
			&\geq\deg(\mathcal{I}_W\mathcal{O}_{\widetilde{C}_1}),
		\end{aligned}
		\]
		as claimed.
	\end{proof}

	\begin{lemma}\label{LemmaDnef}
		Let $D$ be a member of the linear system $|\xi+2H|$ in $X$ as in Lemma \ref{LemmaE(2)}. Then for every integral curve $C\ne \mathbf{l}_0$ in $X$, we have
		\[
		D\cdot C\ge \mu(C).
		\] 
	\end{lemma}
	\begin{proof}
		Let $\mathcal{C}_1$, $\mathcal{C}_2$ be the disjoint conics contained respectively in the planes $\mathcal{P}_1,\mathcal{P}_2$ in  $\mathbb{P}^3$ such that $D$ is realized by $\pi|_D$ as the blow-up of $\mathbb{P}^3$ along $\mathcal{C}_1\cup\mathcal{C}_2$.
		Consider $\Lambda_i$ the strict transform in $D$ of the plane $\mathcal{P}_i$, $i=1,2$. Recall that $\mathcal{P}_1\cap\mathcal{P}_2=L_0$ and $\Lambda_1\cap \Lambda_2=\mathbf{l}_0$.
		
		Now, let $C\ne\mathbf{l}_0$ be an integral curve. 
		Suppose $C\nsubseteq D$. Then by Lemma \ref{LemmaLengthInclusions}, \[\mu(C)=-\deg(\mathcal{I}_{\mathbf{l}_0}\mathcal{O}_{\tilde{C}})\le -\deg(\mathcal{I}_D\mathcal{O}_{\tilde{C}})=D\cdot C,\]
		where $\tilde{C}$ is the normalization of $C$. 
		
		Suppose $C\subset D$. Then it cannot be contained in both $\Lambda_1$ and $\Lambda_2$. Suppose $C\nsubseteq \Lambda_2$, for example. Then, since $\mathbf{l}_0\subset \Lambda_2$, we can conclude as above that $\Lambda_2\cdot C\ge \mu(C)$. Recall that 
		\[
		D|_D\equiv (2P-E_1)+(P-E_2)
		\] 
		(see Lemma \ref{LemmaD} and notation therein). Since $(2P-E_1)$ is nef and $P-E_2\equiv \Lambda_2$, we then conclude that $D\cdot C\ge \mu(C)$.      
	\end{proof}

	\begin{proposition}\label{PropX+}
		The extremal ray $\mathbb{R}_{\ge 0}[\mathbf{l}_0]$ in $\NE(X)$ is contracted by a small extremal contraction $\phi\colon X\to W$. Such extremal contraction induces a flip $X\dashrightarrow X^{+}$, where $X^{+}$ is a smooth Fano 4-fold, with $c_1(X^+)^4=289$, $c_1(X^+)^2c_2(X^+)=166$,  $h^0(-K_{X^+})=63$, $h^{2,2}(X^+)=3$, and $h^{i,j}(X^{+})=0$ for all $i\ne j$.  
	\end{proposition}
	\begin{proof}
		Consider $T:=\xi+3H$. By Lemma \ref{LemmaTnef}, $T$ is nef, and it is not difficult to prove that $T^4>0$, so that $T$ is also big. For $m\in\mathbb{Z}$ big enough, $|mT|$ induces a morphism $g\colon X\to W$, that is the extremal contraction associated to the ray $\mathbb{R}_{\ge 0}[\mathbf{l}_0]$. If $C$ is an irreducible curve contracted by $g$, then $C\equiv\alpha\mathbf{l}_0$ for some $\alpha>0$, and $-K_{X}\cdot C=-\alpha<0$. But then, by Proposition \ref{PropositionNegCurves}, $C$ can only coincide with $\mathbf{l}_0$. It follows that $\Exc(g)=\{\mathbf{l}_0\}$, and $g$ is a small contraction. 
		
		Now, consider $\sigma\colon \hat{X}\to X$ the blowup of $\mathbf{l}_0$. The exceptional divisor $E$ of $\sigma$ coincides with $\mathbb{P}_{\mathbf{l}_0}(\mathcal{O}_{\mathbf{l}_0}(1)^{\oplus 3})\simeq \mathbb{P}^1\times\mathbb{P}^2$. Now, $E|_E=\mathcal{O}_{E}(-1)\simeq\mathcal{O}_{\mathbb{P}^1\times\mathbb{P}^2}(-1,-1)$, and $-K_{\hat{X}}|_E=-K_{E}+E|_E= \mathcal{O}_{\mathbb{P}^1\times\mathbb{P}^2}(2,3)\otimes\mathcal{O}_{\mathbb{P}^1\times\mathbb{P}^2}(-1,-1)\simeq\mathcal{O}_{\mathbb{P}^1\times\mathbb{P}^2}(1,2)$. Hence $-K_{\hat{X}}|_E$ is ample. Observe that $E=\Exc(\phi\circ\sigma\colon \hat{X}\to W)$. Call $L_1$ the class of $\{*\}\times h$, where $h$ is a line in $\mathbb{P}^2$, and $L_2$ the class of $\mathbb{P}^1\times\{*\}$ in $N_1(E)$. Those two classes are clearly independent, also when seen in $N_1(\hat{X})$. Moreover, we have $\NE(\phi\circ\sigma)=\mathbb{R}_{\ge 0}L_1+\mathbb{R}_{\ge 0}L_2$, which is an extremal two-dimensional face of $\NE(\hat{X})$. It follows that the ray spanned by $L_2$ is also extremal, and since it is $K$-negative, it is contracted by a morphism $\sigma^{+}\colon \hat{X}\to X^{+}$, that is a divisorial contraction,  $\Exc(\sigma^{+})=E$ and $\sigma^{+}|_E\colon E\to \mathbb{P}^2$ coincides with the projection onto the second factor of $\mathbb{P}^1\times\mathbb{P}^2$. Denote by $\Pi$ the image of $E$ through $\sigma^+$, which is isomorphic to $\mathbb{P}^2$. By \cite[Theorem 2.3]{An85}, $X^{+}$ is smooth and $\sigma^+$ is the blow-up of $X^+$ along the smooth locus $\Pi$;  $X^{+}$ is then the flip of $X$ with respect to $g$. 
		
		Now we prove that $X^{+}$ must be Fano. For any Cartier divisor $M$ on $X$ and its transform $M^+$ in $X^+$, the pull-backs through $\sigma$ and $\sigma^+$ respectively agree away from $E$. Intersecting their difference with $\mathbb{P}^1\times\{q\}$ gives
		\begin{align}\label{EquationM^+}
			(\sigma^+)^*M^+=\sigma^*M+(M\cdot \mathbf{l}_0)E.
		\end{align}
		Consequently, for $C\ne \mathbf{l}_0$,
		\[
		M^+\cdot C^+=M\cdot C+(M\cdot \mathbf{l}_0)\mu(C),
		\]
		where $C^+$ is the transform of $C$ in $X^+$. 
		and for a line $h\subset \Pi$,
		\[
		M^+\cdot h=-M\cdot \mathbf{l}_0.
		\]
		For $M=D$, equation (\ref{EquationM^+}) and Lemma \ref{LemmaDnef} give 
		\begin{align*}
			D^+\cdot C^+\ge 0 \qquad D^+\cdot h=1.
		\end{align*}
		Every integral curve in $X^+$ is either contained in $\Pi$ or is the transform of a curve in $X$. Hence $D^+$ is nef. The transform $T^+$ of $T\in|\xi+3H|$ is nef as the pullback of an ample divisor from the common contraction target. Finally,
		\[
		-K_{X^+}=D^++T^+.
		\]  
		The two nef classes are linearly independent in the two-dimensional space $N^1(X^+)$. Their positive sum lies in the interior of the nef cone, hence it is ample. This proves that $X^+$ is Fano.

		Let $\phi\colon X\dashrightarrow X^+$ be the resulting flip. We have $X\setminus\dom(\phi)=\{\mathbf{l}_0\}$, and $X^+\setminus\dom(\phi^{-1})=\{\Pi\}$, where $\Pi\simeq\mathbb{P}^2$ is an \emph{exceptional plane}, with $\mathcal{N}_{\Pi/X^+}\simeq \mathcal{O}_{\mathbb{P}^2}(-1)^{\oplus 2}$. Recall that $\Pi$ coincides with $\sigma^{+}(E)\subset X^{+}$.  
		
		Let us now compute the invariants of $X^+$. 
		
		Since $-K_X\equiv2\xi+5H$, as in the proof of Lemma \ref{LemmaTnef}, we can compute $(-K_{X})^4=288$. Moreover, we can use the following well known formula (see for example \cite[Proposition 3.4]{Sec23}) for a smooth $\mathbb{P}^1$-bundle onto a 3-fold $\mathbb{P}_Z(\mathcal{E})\to Z$:
		\begin{align*}
			\chi(\mathcal{O}(-K_{\mathbb{P}(\mathcal{E})}))=\chi(\mathcal{O}_{\mathbb{P}(\mathcal{E})})+6K_Z\cdot c_2(\mathcal{E})-\frac{1}{2}(3K_Z^3+3K_Z\cdot c_1(\mathcal{E})^2)-\frac{1}{3}K_Z\cdot c_2(Z)
		\end{align*}
		and compute $\chi(\mathcal{O}(-K_{X}))=63$.

		Now, by \cite[Corollary 3.10]{Cas17}, since $X^+$ is Fano and $X^+\setminus\dom(\phi^{-1})=\{\Pi\}$ has one connected component, we have $(-K_{X}^+)^4=(-K_{X})^4+1=289$ and $\chi(\mathcal{O}(-K_{X^+}))=\chi(\mathcal{O}(-K_{X}))=63$. 
		Using the Hirzebruch-Riemann-Roch theorem, one can additionally compute $c_1(X^+)^2c_2(X^+)=12h^0(-K_{X^+})-2c_1(X^+)^4-12=166$. 
		
		As for the Hodge numbers, one can easily check that $h^{2,2}(X^{+})=3$ and $h^{i,j}(X^+)=0$ for all $i\ne j$ by applying, for example, \cite[Lemma 3.1]{Sec23}. 
		
	\end{proof}
	
	This concludes the proof of Theorem \ref{TheoremNonFanoCase}.
	
	\subsection{The second contraction of $X^+$}
	In the previous sections, we proved that there exists a flip $\mathbb{P}(\mathcal{E})=X\dasharrow X^+$ to a Fano variety $X^+$ with Picard number $\rho_{X^+}=2$. In this section we are going to describe the second extremal contraction of $X^+$, i.e. the one that is not related to the flip.  
	\begin{tcolorbox}[
		title={Roadmap},
		colback=gray!5,
		colframe=black!60,
		boxrule=0.5pt,
		arc=1mm,
		left=2mm,
		right=2mm,
		top=1mm,
		bottom=1mm
		]
		\begin{itemize}
			\item We prove that curves in $X$ with anticanonical degree one are in $1:1$ correspondence with lines in $\mathbb{P}^3$ that are tangent to $Z$ at precisely one point. 
			\item  The divisor $F\in|\xi+H|$ can be described as the blow-up of $\mathbb{P}^3$ along a non-reduced schematic structure supported on $L_0$. We study the geometry of the strict transforms $\Lambda_{\alpha,\beta}$ in $F$ of the planes containing $L_0$. We also note that every such transform is contained in some divisor $D$ in $|\xi+2H|$, and we shed light on its relation with the geometry of $D$ itself.
			\item  We study the geometry of $\hat{X}$, the blow-up of $X$ along $\mathbf{l}_0$, in particular we describe the strict transforms $\hat{\Lambda}_{\alpha,\beta}$ of $\Lambda_{\alpha,\beta}$. 
			\item  We study the geometry of the transforms $\Lambda_{\alpha,\beta}^+$ of the surfaces $\Lambda_{\alpha,\beta}$ under the flip $X\dasharrow X^+$, which we describe in Section \ref{SectionX+}, by studying how the contraction $\sigma^+\colon \hat{X}\to X^+$ acts on $\hat{\Lambda}_{\alpha,\beta}$. 
			\item We prove that those transforms $\Lambda_{\alpha,\beta}^+$ are precisely the 2-dimensional fibers of the second contraction of $X^+$, which is a divisorial contraction of type $(3,1)$. The transform $F^+$ of $F$ in $X^+$ coincides with the exceptional divisor of such contraction.    
		\end{itemize}
	\end{tcolorbox}

	\paragraph{Curves in $X$ with anticanonical degree one.}
	We now further explore the geometry of $X$, starting with the curves in $X$ that have anticanonical degree one.

	\begin{lemma}\label{LemmaDeg1}
		There is a $1:1$ correspondence
		\begin{align*}
			\left\{
			\begin{array}{c}
				\text{Integral curves }\tilde L\subset X,\quad
				-K_X\cdot\tilde L=1,\\[-2pt]
				\pi(\tilde L)\neq L_0
			\end{array}
			\right\}
			\quad\rightleftarrows\quad
			\left\{
			\begin{array}{c}
				\text{Lines }L\subset\mathbb P^3,\\[-2pt]
				\text{tangent to }Z\text{ at one point}
			\end{array}
			\right\}
		\end{align*}
		under which $\tilde{L}\subset X$ corresponds to its image $\pi(\tilde{L})$ in $\mathbb{P}^3$, and $L\subset \mathbb{P}^3$ corresponds to the negative section of $\pi^{-1}(L)\simeq \mathbb{F}_3$ in $X$.  
	\end{lemma}
	
	\begin{proof}
		By Proposition \ref{PropositionJumping}, given a line $L\subset \mathbb{P}^3$, $L$ is tangent to $Z$ at one point if and only if $V:=\pi^{-1}(L)\simeq \mathbb{P}(\mathcal{O}_{\mathbb{P}^1}(1)\oplus\mathcal{O}_{\mathbb{P}^1}(-2))\simeq\mathbb{F}_3$.  
		In that case, let $\Gamma^{-}$ be the negative section of $V$. By Lemma \ref{LemmaRmk2.5}, $-K_{X}\cdot \Gamma^-=-K_{\mathbb{P}^3}\cdot L-3=4-3=1$. 
		
		Conversely, let $\tilde{L}$ be an integral curve in $X$ such that $-K_X\cdot \tilde{L}=1$. Consider $L:=\pi(\tilde{L})$, with $d:=\deg(L)\ge 1$, and suppose $L\ne L_0$. We first prove that $d=1$ and $L$ is a line. 
		By Lemma \ref{LemmaDnef}, $D\cdot \tilde{L}\ge 0$. Therefore the fact that
		\begin{align*}
			1=-K_X\cdot C=(2D+H)\cdot C
		\end{align*} 
		and the nefness of $H$ imply that $D\cdot C=0$ and $H\cdot C=1$.		
		This concludes the proof that $L=\pi(\tilde{L})$ must be a line. 
		Set $V:=\pi^{-1}(L)\simeq \mathbb{F}_e$. By Proposition \ref{PropositionJumping}, $e\in\{1,3,5\}$. Let $\alpha,\beta\in \mathbb{Z}_{\ge 0}$ be coefficients such that $\tilde{L}\sim \alpha \Gamma^{-}+\beta\mathbf{f}$ in $V$, where $\mathbf{f}$ is a fiber of the $\mathbb{P}^1$-bundle $V\to L$. If $\alpha=0$, then $-K_{X}\cdot \tilde{L}=2\beta$ is even, a contradiction. Therefore $\alpha\ge 1$. Suppose $\tilde{L}\ne\Gamma^{-}$. Then by Proposition 2.21 \cite{Har77}, $\alpha=1$, $\beta\ge 1$ or $\alpha\ge 2$, $\beta\ge \alpha e$. Observe that $-K_X\cdot \tilde{L}=(4d-e)\alpha+2\beta=4d\alpha+(2\beta-\alpha e)$. If $\alpha\ge 2$, it follows that $-K_X\cdot \tilde{L}\ge 4d\alpha\ge 4$, a contradiction. If $\alpha=1$ and $\beta\ge 1$, then $-K_X\cdot \Gamma^{-}=4-e<4-e+2\beta=-K_X\cdot \tilde{L}=1$. By Proposition \ref{PropositionNegCurves}, it follows that $-K_X\cdot \Gamma^{-}=-1$, hence $e=4+1=5$. But then $\alpha=\beta=1$, and $L=L_0$. On the other hand, if $e\in\{1,3\}$ it follows that $\tilde{L}$ must coincide with $\Gamma^{-}$. Hence we have $1=-K_X\cdot \tilde{L}=4-e$, which implies $e=4-1=3$, and we deduce that $L$ is a line that is tangent to $Z$.       
	\end{proof}
	
	\begin{corollary}\label{CorollaryVandConics}
		For every $[\alpha:\beta]\in\mathbb{P}^1$, consider the plane $\mathcal{P}_{\alpha,\beta}=\{\alpha x+\beta y=0\}$ in $\mathbb{P}^3$. Let $D_{\alpha,\beta}^{\alpha',\beta'}$ be a member of the linear system $|\xi+2H|$ in $X$ as in Lemma \ref{LemmaE(2)}, for some $[\alpha':\beta']\ne[\alpha:\beta]$. Recall that the restriction of $\pi\colon X\to \mathbb{P}^3$ to $D_{\alpha,\beta}^{\alpha',\beta'}$ realizes it as the blow-up of two disjoint conics  $\mathcal{C}_{\alpha,\beta}\subset \mathcal{P}_{\alpha,\beta}$ and $\mathcal{C}_{\alpha',\beta'}\subset\mathcal{P}_{\alpha',\beta'}$. Then we have 
		\[
		V(\beta f+\alpha g)\cap L_0=\mathcal{P}_{\alpha,\beta}\cap\mathcal{C}_{\alpha',\beta'}
		\]
		where $f,g$ are the quadratic forms appearing in Proposition \ref{PropositionZ}. 
	\end{corollary}
	\begin{proof}
		Fix $[\alpha:\beta]\ne [\alpha',\beta']$ in $\mathbb{P}^1$, and adopt the following simplified notation for the objects described in the statement:
		\begin{itemize}
			\item $D=D_{\alpha,\beta}^{\alpha',\beta'}$;
			\item $\mathcal{P}=\mathcal{P}_{\alpha,\beta}$,
			$\mathcal{P}'=\mathcal{P}_{\alpha',\beta'}$;
			\item $\mathcal{C}=\mathcal{C}_{\alpha,\beta}$, $\mathcal{C}'=\mathcal{C}_{\alpha',\beta'}$;
			\item $V=V(\beta f+\alpha g)\cap L_0$.
		\end{itemize}  
		Let $\ell$ be the transform in $D$ of any line $L\subset \mathcal{P}$ passing through a point in $\mathcal{C}'\cap \mathcal{P}$. Then by Lemma \ref{LemmaD}, we know that $-K_X\cdot \ell =1$, and by Lemma \ref{LemmaDeg1}, $L$ is tangent to $Z$, meaning that $L\cap L_0\subseteq V$ by Remark \ref{RemarkTangentLines}. This proves that $\mathcal{C'}\cap \mathcal{P}\subseteq V$. 
		
		On the other hand, suppose that there exists a point in $V\setminus(\mathcal{C}'\cap \mathcal{P})$, and consider a line $L\subset \mathcal{P}$ passing through it. By Remark \ref{RemarkTangentLines}, $L$ is tangent to $Z$, and Lemma \ref{LemmaDeg1} implies that $\pi^{-1}(L)\simeq\mathbb{F}_3$ and the negative section $\ell \subset \pi^{-1}(L)$ has anticanonical degree 1 in $X$. Then we have:
		\begin{align*}
			1=-K_X\cdot \ell=(2\xi+5H)\cdot \ell\qquad\qquad 1=H\cdot \ell
		\end{align*}
		which implies $(\xi+2H)\cdot \ell=0$. Therefore either $D\cap \ell=\emptyset$ or $\ell\subset D$. Now, $L\cap \mathcal{C}\ne \emptyset$ by Bézout, thus $\ell\cap \pi^{-1}(\mathcal{C})\ne \emptyset $. Since $\pi^{-1}(\mathcal{C})\subset D$, we get $D\cap \ell\ne \emptyset$, thus $\ell\subset D$. In Lemma \ref{LemmaD} we proved that the only irreducible curves in $D$ that have anticanonical degree $1$ are the ones that are realized as the transforms in $D$ of lines in $\mathcal{P}$ passing through one point in $\mathcal{C}'\cap\mathcal{P}$. Therefore $L\cap (V\setminus (\mathcal{C}'\cap\mathcal{P}))=L\cap L_0\subseteq \mathcal{C}'\cap\mathcal{P}$, a contradiction.
		
		Therefore 
		\[
		V=\mathcal{C}'\cap \mathcal{P}.
		\]   
	\end{proof}

	\paragraph{The family of surfaces $\Lambda_{\alpha,\beta}$ in $X$.}	
	The family of degree-one curves described in Lemma \ref{LemmaDeg1} sweeps out certain surfaces $\Lambda_{\alpha,\beta}$. We now identify these surfaces explicitly, and study their behavior under the blow-up of $\mathbf{l}_0$.

	\begin{lemma}\label{LemmaHDHF}
		For every $[\alpha:\beta]\in\mathbb{P}^1$, consider the plane $\mathcal{P}_{\alpha,\beta}=\{\alpha x+\beta y=0\}$ in $\mathbb{P}^3$. Let $D_{\alpha,\beta}^{\alpha',\beta'}$ be a member of the linear system $|\xi+2H|$ in $X$ as in Lemma \ref{LemmaE(2)}, for some $[\alpha':\beta']\ne[\alpha:\beta]$. Consider $F\in|\xi+H|$ as in Lemma \ref{LemmaE(1)}. 
		Recall that the restrictions of $\pi\colon X\to \mathbb{P}^3$ to $D_{\alpha,\beta}^{\alpha',\beta'}$ and $F$ realizes them as blow-ups of subschemes of $\mathbb{P}^3$. Consider the following surfaces in $X$. 
		\begin{itemize}
			\item The transform $\Lambda_{\alpha,\beta}^D$ of $\mathcal{P}_{\alpha,\beta}$ in $D_{\alpha,\beta}^{\alpha',\beta'}$. 
			\item The transform $\Lambda_{\alpha,\beta}^F$ of $\mathcal{P}_{\alpha,\beta}$ in $F$. 
		\end{itemize} 
		We have $\Lambda_{\alpha,\beta}^D=\Lambda_{\alpha,\beta}^F$. 
	\end{lemma}
	\begin{proof}
		Recall that, by Corollary \ref{CorollaryDeg1}, $\Lambda_{\alpha,\beta}^D$ has an open dense subset $U_{\alpha,\beta}$ such that for every point $p\in U_{\alpha,\beta}$, there exists a curve $C\subset \Lambda_{\alpha,\beta}^D$ such that $p\in C$, $-K_X\cdot C=1$ and $C$ projects onto a line in $\mathbb{P}^3$. Let $C$ be one of those curves. Then 
		\begin{align*}
			(\xi+H)\cdot C= \frac{1}{2}((2\xi+5H)\cdot C - 3H\cdot C)=\frac{1}{2}((-K_X)\cdot C - 3H\cdot C)= \frac{1}{2}(1 - 3)=-1 
		\end{align*}
		which implies that $C\subset F$. It follows that $\Lambda_{\alpha,\beta}^D\subseteq F$. 
		
		On $\mathcal{P}_{\alpha,\beta}$, consider the open subset given by $W:=\mathcal{P}_{\alpha,\beta}\setminus L_0$. Clearly the restriction of the blow-up $\pi|_F$ is the identity on $\pi^{-1}(W)\cap F$, so that 
		\begin{align*}
			(\Lambda_{\alpha,\beta}^D\setminus \pi^{-1}(L_0))\subseteq \pi^{-1}(W)\cap F = \Lambda_{\alpha,\beta}^F\setminus \pi^{-1}(L_0)
		\end{align*} 
		which implies that  $\Lambda_{\alpha,\beta}^D=\Lambda_{\alpha,\beta}^F$ in $F$. 
	\end{proof}

	\begin{notations}
		In light of Lemma \ref{LemmaHDHF}, for any fixed $[\alpha:\beta]\in\mathbb{P}^1$, we will adopt the same notation $\Lambda_{\alpha,\beta}$ for both $\Lambda_{\alpha,\beta}^D$ and $\Lambda_{\alpha,\beta}^F$, because they refer to the same subvariety in $X$. 
	\end{notations}

	\begin{lemma}\label{LemmaHGeometry}
		Let $[\alpha:\beta]\in\mathbb P^1$. If $\alpha\ne0$, then $\Lambda_{\alpha,\beta}$ is realized by
		$\pi|_{\Lambda_{\alpha,\beta}}$ as the blow-up of
		\(
		\mathcal P_{\alpha,\beta}\simeq\mathbb P^2_{[y:z:w]}
		\)
		along the ideal
		\(
		J=(y,\beta f+\alpha g)\subset\mathbb C[y,z,w].
		\)
		If $\beta\ne0$, the analogous statement holds with
		\(
		\mathcal P_{\alpha,\beta}\simeq\mathbb P^2_{[x:z:w]}
		\)
		and
		\(
		J=(x,\beta f+\alpha g)\subset\mathbb C[x,z,w].
		\)
		In particular:
		\begin{itemize}
			\item[(a)] If $(\beta f+\alpha g)|_{L_0}$ is not a perfect square, then
			\(
			\Lambda_{\alpha,\beta}\simeq\Bl_{2\mathrm{pts}}\mathbb P^2.
			\)
			More precisely, $J$ defines two distinct reduced points on $L_0$. The three $(-1)$-curves on $\Lambda_{\alpha,\beta}$ are the two exceptional curves of the blow-up over
			\(
			V(\beta f+\alpha g)\cap L_0
			\)
			and the strict transform of $L_0$ in $F$, namely $\mathbf l_0$.

			\item[(b)] If $(\beta f+\alpha g)|_{L_0}$ is a perfect square, then $J$ defines a non-reduced zero-dimensional subscheme of length two supported at a single point of $L_0$. In this case, $\Lambda_{\alpha,\beta}$ is isomorphic to the blow-up of a quadric cone at a smooth point. Its unique $(-1)$-curve is the strict transform of $L_0$ in $F$, namely $\mathbf l_0$.
		\end{itemize}
		Moreover, $(\beta f+\alpha g)|_{L_0}$ is a perfect square for exactly two choices of $[\alpha:\beta]\in\mathbb P^1$.
	\end{lemma}
	
	\begin{proof}
		Since $\Lambda_{\alpha,\beta}$ is the strict transform of
		$\mathcal P_{\alpha,\beta}$ in
		\(
		F\simeq\Bl_{\mathcal{I}_Z}\mathbb P^3,
		\)
		it follows that $\pi|_{\Lambda_{\alpha,\beta}}$ realizes $\Lambda_{\alpha,\beta}$ as
		\(
		\Bl_{\mathcal{I}_Z\mathcal O_{\mathcal P_{\alpha,\beta}}}
		\mathcal P_{\alpha,\beta}.
		\)
		
		Suppose first that $\alpha\ne0$. Using coordinates
		$[y:z:w]$ on
		$\mathcal P_{\alpha,\beta}\simeq\mathbb P^2$, we have
		\[
		\mathcal{I}_Z\mathcal O_{\mathcal P_{\alpha,\beta}}
		=
		(y^2,y(\beta f+\alpha g)).
		\]
		Put
		\[
		q:=\beta f+\alpha g.
		\]
		Then
		\[
		(y^2,yq)=y(y,q).
		\]
		Since $\mathbb C[y,z,w]$ is a domain, $y$ is a non-zero-divisor. Hence multiplication by $y^n$ induces, for every $n\ge0$, an isomorphism
		\[
		(y,q)^n\xrightarrow{\sim}(y^2,yq)^n,
		\qquad
		h\longmapsto y^n h.
		\]
		These isomorphisms are compatible with multiplication and therefore induce an isomorphism of graded Rees algebras
		\[
		\mathcal R((y,q))
		\simeq
		\mathcal R((y^2,yq)).
		\]
		Consequently,
		\[
		\Bl_{(y^2,yq)}\mathbb P^2
		\simeq
		\Bl_{(y,q)}\mathbb P^2
		\]
		over $\mathbb P^2$. Thus $\Lambda_{\alpha,\beta}$ is the blow-up of
		$\mathcal P_{\alpha,\beta}$ along
		\[
		J=(y,\beta f+\alpha g).
		\]
		
		The case $\beta\ne0$ is analogous, using coordinates
		$[x:z:w]$ on $\mathcal P_{\alpha,\beta}$.
		
		Now assume that $(\beta f+\alpha g)|_{L_0}$ is not a perfect square. Since $J$ contains $y$, its zero-dimensional subscheme is contained in
		$L_0=\{y=0\}$. Its residual equation on $L_0$ is
		\[
		(\beta f+\alpha g)|_{L_0}=0.
		\]
		If this quadratic form is not a square, it has two distinct zeros on $L_0$. Hence $J$ defines two distinct reduced points, and therefore
		\[
		\Lambda_{\alpha,\beta}\simeq\Bl_{2\mathrm{pts}}\mathbb P^2.
		\]
		The two exceptional curves are the fibers of $\pi|_{\Lambda_{\alpha,\beta}}$ over these two points. The strict transform $\mathbf l_0$ of $L_0$ is the third $(-1)$-curve.
		
		If instead $(\beta f+\alpha g)|_{L_0}$ is a perfect square, then its zero scheme consists of a single point with multiplicity two. Thus $J$ defines a non-reduced zero-dimensional subscheme of length two supported at that point. The resulting surface is the one described in the statement, namely the blow-up of a quadric cone at a smooth point. In this case the strict transform $\mathbf l_0$ of $L_0$ is the unique $(-1)$-curve.
		
		The final assertion follows from Remark~\ref{RemarkTangentLines}.

	\end{proof}
	
	\begin{remark}
		For every point $p\in L_0$, there exists a unique
		$[\alpha:\beta]\in\mathbb P^1$ such that
		\(
		p\in V(\beta f+\alpha g)\cap L_0.
		\)
		Moreover,
		\(
		V(\beta f+\alpha g)\cap L_0
		=
		V(\beta' f+\alpha' g)\cap L_0
		\)
		if and only if
		\(
		[\alpha:\beta]=[\alpha':\beta'].
		\)
	\end{remark}
	
	\paragraph{The geometry of $\hat{X}=\Bl_{\mathbf{l}_0}X$}
	
	We are now going to investigate the geometry of the blowup $\hat{X}:=\Bl_{\mathbf{l}_0}X$. 
	
	Let $\sigma\colon \hat{X}\to X$ be the blow-up morphism. If not otherwise specified, we will denote by $E$ the exceptional divisor of $\sigma$ throughout the rest of this section. Notice that, by Proposition \ref{PropositionNegCurves} \begin{align*}E\simeq\mathbb{P}_{\mathbf{l}_0}(\mathcal{N}_{\mathbf{l}_0/X}^{*})\simeq\mathbb{P}_{\mathbf{l}_0}(\mathcal{O}_{\mathbf{l}_0}(1)^{\oplus 3})\simeq \mathbb{P}^1\times\mathbb{P}^2. 
	\end{align*}
	\begin{lemma}\label{LemmaS_0}
		Consider $S_0:=\pi^{-1}(L_0)\subset X$. Let $\sigma\colon \hat{X}\to X$ be the blowup of $X$ along $\mathbf{l}_0$. Let $E$ be the exceptional divisor of $\sigma$, and denote by $\hat{S}_0$ the strict transform of $S_0$ in $\hat{X}$. Then $\hat{S}_0\cap E$, when seen in $E\simeq \mathbb{P}^1\times\mathbb{P}^2$, is a curve of bidegree $(1,4)$. 
	\end{lemma}
	\begin{proof}
		We know that $S_0\simeq\mathbb{F}_5$ and $\mathbf{l}_0$ is the negative section of $S_0$, i.e. $\mathbf{l}_0^2=-5$ in $S_0$. This implies that $\mathcal{N}_{\mathbf{l}_0/S_0}\simeq\mathcal{O}_{\mathbf{l}_0}(-5)$. Now, the natural immersion $\mathcal{O}_{\mathbf{l}_0}(-5)\simeq\mathcal{N}_{\mathbf{l}_0/S_0}\hookrightarrow \mathcal{N}_{\mathbf{l}_0/X}\simeq\mathcal{O}_{\mathbf{l}_0}(-1)^{\oplus 3}$ determines the inclusion $\hat{S}_0\cap E\simeq \mathbb{P}_{\mathbf{l}_0}(\mathcal{N}_{\mathbf{l}_0/S_0}^{*})\hookrightarrow \mathbb{P}_{\mathbf{l}_0}(\mathcal{N}_{\mathbf{l}_0/X}^{*})\simeq E$. The image of such inclusion is therefore the graph of a morphism 
		\begin{align*}
			\mathbb{P}^1&\longrightarrow\mathbb{P}^2\\
			[s:t]&\mapsto [\phi_1(s,t):\phi_2(s,t):\phi_3(s,t)]
		\end{align*}
		where $\phi_i(s,t)$ are homogeneous quartic polynomials.
		This proves that $\hat{S}_0\cap E$ is a curve of bidegree $(1,4)$ in $E\simeq\mathbb{P}^1\times\mathbb{P}^2$.  
	\end{proof}
	
	The following proposition will be particularly useful in the proof of Proposition \ref{PropositionContractionX+}.
	
	\begin{proposition}\label{PropositionBlowup}
		Let $\mathcal{P}_{\alpha,\beta}=\{\alpha x+\beta y=0\}$ be any plane containing $L_0=\{x=y=0\}$ in $\mathbb{P}^3$, and consider $\Lambda_{\alpha,\beta}$, the strict transform of $\mathcal{P}_{\alpha,\beta}$ in $F$.  Let $\hat{\Lambda}_{\alpha,\beta}$ be the strict transform of $\Lambda_{\alpha,\beta}$ in $\hat{X}$. Then $\hat{\Lambda}_{\alpha,\beta}\cap E$ coincides with $\mathbb{P}^1\times\{q_{\alpha, \beta}\}\subset \mathbb{P}^1\times\mathbb{P}^2$ when seen in $E\simeq\mathbb{P}^1\times\mathbb{P}^2$, where $q_{\alpha,\beta}$ is a point in $\mathbb{P}^2$, such that $[\alpha:\beta]=[\tilde{\alpha}:\tilde{\beta}]$ if and only if $q_{\alpha,\beta}=q_{\tilde{\alpha},\tilde{\beta}}$.

		Besides, $\hat{\Lambda}_{\alpha,\beta}$ is always isomorphic to $\Lambda_{\alpha,\beta}$ through $\sigma$, and there are only two possibilities:
		\begin{itemize}
			\item[(a)] If $\Lambda_{\alpha,\beta}\simeq\Bl_{2pts}\mathbb{P}^2$, then we have: 
			\begin{itemize}
				\item $\hat{\Lambda}_{\alpha,\beta}\cap E$, when seen as a curve in $\hat{\Lambda}_{\alpha,\beta}\simeq \Bl_{2pts}\mathbb{P}^2$, coincides with the strict transform of the line passing through the two blown-up points.
				\item $\hat{\Lambda}_{\alpha,\beta}\cap \hat{S}_0$ coincides with the transforms in $\hat{X}$ of the two one-dimensional fibers of $\pi\colon X\to \mathbb{P}^3$ above $V(\beta f+\alpha g)\cap L_0$. When seen in $\hat{\Lambda}_{\alpha,\beta}\simeq\Bl_{2pts}\mathbb{P}^2$, $\hat{\Lambda}_{\alpha,\beta}\cap\hat{S}_0$ coincides with the two exceptional divisors above the two blown-up points. 
				\item  The intersection $\hat{S}_0\cap E\cap \hat{\Lambda}_{\alpha,\beta}$ consists of two points, whose images through $\pi\circ\sigma$ coincide with the two points in $V(\beta f+\alpha g)\cap L_0$ in $\mathbb{P}^3$.
			\end{itemize} 
			\item[(b)] If $\Lambda_{\alpha,\beta}$ is isomorphic to the blow-up of a quadric cone at a smooth point, then we have:
			\begin{itemize}
				\item $\hat{\Lambda}_{\alpha,\beta}\cap E$ coincides with the $(-1)$-curve in $\hat{\Lambda}_{\alpha,\beta}$, disjoint from the singular point, whose contraction realizes $\Lambda_{\alpha,\beta}$ as the blow-up of a quadric cone at a smooth point.  
				\item $\hat{\Lambda}_{\alpha,\beta}\cap \hat{S}_0$ coincides with the transform of the one-dimensional fiber of $\pi\colon X\to \mathbb{P}^3$ above $V(\beta f+\alpha g)\cap L_0$, and contains the singular point of $\hat{\Lambda}_{\alpha,\beta}$. 
				\item The intersection $\hat{S}_0\cap E\cap \hat{\Lambda}_{\alpha,\beta}$ consists of one point, whose image through $\pi\circ\sigma$ coincides with $V(\beta f+\alpha g)\cap L_0$ in $\mathbb{P}^3$.
			\end{itemize} 
		\end{itemize} 
	\end{proposition}
	\begin{proof}
		By Lemma \ref{LemmaHGeometry}, the transform $\Lambda_{\alpha,\beta}$ of $\mathcal{P}_{\alpha,\beta}$ in $F\simeq\Bl_{\mathcal{I}_Z}\mathbb{P}^3$ is isomorphic to the blow-up of $\mathcal{P}_{\alpha,\beta}$ along the ideal $\mathcal{I}_Z\mathcal{O}_{{P}_{\alpha,\beta}}$, which yields two cases. 
		\begin{itemize}
			\item[(a)] For a general $[\alpha:\beta]\in\mathbb{P}^1$, we get $\Lambda_{\alpha,\beta}\simeq \Bl_{2pts}\mathbb{P}^2$. More precisely, in this case $\Lambda_{\alpha,\beta}$ is realized through $\pi$ as the blow-up of $\mathcal{P}_{\alpha,\beta}$ along the two distinct points in $V(\beta f+\alpha g)\cap L_0$. Observe that $\Lambda_{\alpha,\beta}$ contains $\mathbf{l}_0$ as the $(-1)$-curve realized as the transform of the line in $\mathbb{P}^2$ passing through the two blown-up points. Since $\mathbf{l}_0$ is a codimension 1 smooth subvariety of $\Lambda_{\alpha,\beta}$, the transform $\hat{\Lambda}_{\alpha,\beta}$ through the blowup of $\mathbf{l}_0$ is isomorphic to $\Lambda_{\alpha,\beta}$. In particular, $\hat{\Lambda}_{\alpha,\beta}\cap E$ must be a section of the $\mathbb{P}^2$-bundle $E\to \mathbf{l}_0$, and must coincide with the transform of the line passing through the two blown-up points when seen in $\hat{\Lambda}_{\alpha,\beta}\simeq\Bl_{2pts}\mathbb{P}^2$.  
			\item[(b)] If $(\beta f+\alpha g)|_{L_0}$ is a perfect square, we get that $\Lambda_{\alpha,\beta}$ is isomorphic to the blowup of a quadric cone at a smooth point. In that case, the transform $\hat{\Lambda}_{\alpha,\beta}$ is isomorphic to $\Lambda_{\alpha,\beta}$ through $\sigma$ because the blown-up locus $\mathbf{l}_0$ is a smooth codimension 1 subvariety, hence again $\hat{\Lambda}_{\alpha,\beta}\cap E$ must be a section of the $\mathbb{P}^2$-bundle $E\to\mathbf{l}_0$, and coincides with the $(-1)$-curve in $\hat{\Lambda}_{\alpha,\beta}$ whose contraction realizes the surface as the blowup of a quadric cone.  
		\end{itemize}
		
		We are now going to prove that in both cases $\hat{\Lambda}_{\alpha,\beta}\cap E$ coincides with $\mathbb{P}^1\times\{q_{\alpha,\beta}\}$ in $E\simeq \mathbb{P}^1\times\mathbb{P}^2$, where $q_{\alpha,\beta}$ is a point in $\mathbb{P}^2$. From now on, fix any $[\alpha':\beta']\ne [\alpha:\beta]$ and, for simplicity, denote by $D$ the blowup $D_{\alpha,\beta}^{\alpha',\beta'}$ of $\mathbb{P}^3$ along the conics $\mathcal{C}_{\alpha,\beta}\cup\mathcal{C}_{\alpha',\beta'}$. 
		
		First consider the following exact sequence of normal bundles:
		\begin{equation}\label{NormalSeqH}
			\begin{tikzcd}
				0\ar[r]& \mathcal{N}_{\mathbf{l}_0/\Lambda_{\alpha,\beta}}\ar[r,"a"]&\mathcal{N}_{\mathbf{l}_0/D}\ar[r]&\mathcal{N}_{\Lambda_{\alpha,\beta}/D}|_{\mathbf{l}_0}\ar[r]& 0.
			\end{tikzcd}
		\end{equation}
		By studying the geometry of $\Lambda_{\alpha,\beta}$, in both cases $(a)$, $(b)$ one can easily compute $\mathcal{N}_{\mathbf{l}_0/\Lambda_{\alpha,\beta}}\simeq \mathcal{O}_{\mathbf{l}_0}(-1)$. Moreover $\deg(\mathcal{N}_{\Lambda_{\alpha,\beta}/D}|_{\mathbf{l}_0})=\Lambda_{\alpha,\beta}\cdot \mathbf{l}_0$ in $D$. Now, $\Lambda_{\alpha,\beta}$ is equivalent to $P-E_1$ in $N^1(D)$, where $P$ is the transform in $D$ of a general plane in $\mathbb{P}^3$ and $E_1$ is the exceptional divisor above $\mathcal{C}_{\alpha,\beta}$. Hence $\Lambda_{\alpha,\beta}\cdot \mathbf{l}_0=(P-E_1)\cdot\mathit{r_0}=1-2=-1$ (see Lemma \ref{LemmaD} and previous discussions). Hence $\mathcal{N}_{\Lambda_{\alpha,\beta}/D}|_{\mathbf{l}_0}\simeq\mathcal{O}_{\mathbf{l}_0}(-1)$. Since $\Ext^1(\mathcal{O}_{\mathbb{P}^1}(-1),\mathcal{O}_{\mathbb{P}^1}(-1))=0$, the sequence (\ref{NormalSeqH}) splits. 
		
		Now consider the following exact sequence of normal bundles:   
		
		\begin{equation}\label{NormalSeq}
			\begin{tikzcd}
				0\ar[r]& \mathcal{N}_{\mathbf{l}_0/\Lambda_{\alpha,\beta}}\ar[r,"c"]&\mathcal{N}_{\mathbf{l}_0/X}\ar[r]&\mathcal{N}_{\Lambda_{\alpha,\beta}/X}|_{\mathbf{l}_0}\ar[r]& 0.
			\end{tikzcd}
		\end{equation}
		We claim that the above sequence splits so that we obtain $\mathcal{N}_{\Lambda_{\alpha,\beta}/X}|_{\mathbf{l}_0}\simeq \mathcal{O}_{\mathbf{l}_0}(-1)^{\oplus 2}$. Indeed, consider the following commutative diagram: 
		\begin{footnotesize}
			\begin{equation*}
				\begin{tikzcd}
					&&&&&0\\
					0\ar[dr]&&&&\mathcal{N}_{D/X}|_{\mathbf{l}_0}\ar[ur]&\\
					0\ar[r]&\mathcal{N}_{\mathbf{l}_0/\Lambda_{\alpha,\beta}}\ar[rd,"a"]\ar[rr,"c"]&&\mathcal{N}_{\mathbf{l}_0/X}\ar[r]\ar[ur]&\mathcal{N}_{\Lambda_{\alpha,\beta}/X}|_{\mathbf{l}_0}\ar[r]&0\\
					&&\mathcal{N}_{\mathbf{l}_0/D}\ar[ur, "b"]\ar[dr]&&&\\
					&0\ar[ur]&&\mathcal{N}_{\Lambda_{\alpha,\beta}/D}|_{\mathbf{l}_0}\ar[dr]&&&\\
					&&&&0&
				\end{tikzcd}
			\end{equation*}
		\end{footnotesize}
		We know from the proof of Proposition \ref{PropositionNegCurves} that $b$ splits. Moreover $a$ is splitting, by the discussion above. Thus, by the commutativity of the above diagram, we can conclude that $c$ is splitting as well.  
		
		The inclusion $\hat{\Lambda}_{\alpha,\beta}\cap E\hookrightarrow E$ coincides with the morphism \[\mathbb{P}(\mathcal{N}_{\mathbf{l}_0/\Lambda_{\alpha,\beta}}^{*})=\mathbb{P}_{\mathbf{l}_0}(\mathcal{O}_{\mathbf{l}_0}(1))\hookrightarrow \mathbb{P}_{\mathbf{l}_0}(\mathcal{O}_{\mathbf{l}_0}(1)^{\oplus 3})= \mathbb{P}(\mathcal{N}_{\mathbf{l}_0/X}^{*})\]
		induced by (\ref{NormalSeq}). Hence $\hat{\Lambda}_{\alpha,\beta}\cap E\simeq \mathbb{P}^1\times \{q_{\alpha,\beta}\}$ in $E\simeq \mathbb{P}^1\times\mathbb{P}^2$, where $q_{\alpha,\beta}$ is a point in $\mathbb{P}^2$. Notice that, since $\mathcal{N}_{\mathbf{l}_0/\Lambda_{\alpha,\beta}}$ and $\mathcal{N}_{\mathbf{l}_0/\Lambda_{\alpha',\beta'}}$ embed in $\mathcal{N}_{\mathbf{l}_0/X}$ as two distinct direct factors, we have that $\hat{\Lambda}_{\alpha,\beta}\cap E$ is in the same form as $\hat{\Lambda}_{\alpha',\beta'}\cap E$, but $(\hat{\Lambda}_{\alpha,\beta}\cap E)\cap (\hat{\Lambda}_{\alpha',\beta'}\cap E)=\emptyset$, i.e. $q_{\alpha,\beta}\ne q_{\alpha',\beta'}$. Since this argument can be made for any choice of $[\alpha':\beta']\ne[\alpha:\beta]$, we conclude that $q_{\alpha,\beta}=q_{\alpha',\beta'}\iff[\alpha:\beta]=[\alpha':\beta']$. 
		
		As for the intersection $\hat{\Lambda}_{\alpha,\beta}$ with $\hat{S}_0$, suppose we are in case (a). One can see that $S_0\cap \Lambda_{\alpha,\beta}$ is equal to the union of $\mathbf{l}_0$ and the two fibers $f_1,f_2$ of $\pi\colon X\to \mathbb{P}^3$ above the two distinct points in $V(\beta f+\alpha g)\cap L_0$.  It follows that the intersection $\hat{\Lambda}_{\alpha,\beta}\cap\hat{S}_0$ consists of the union of the transforms $\hat{f}_1$, $\hat{f}_2$ of $f_1,f_2$ in $\hat{X}$ and the intersection $(\hat{\Lambda}_{\alpha,\beta}\cap E)\cap (\hat{S}_0\cap E)$ in $E$. Now $\hat{\Lambda}_{\alpha,\beta}\cap E\simeq \mathbb{P}^1\times\{q_{\alpha,\beta}\}$ and $\hat{S}_0\cap E$ is a curve of bidegree $(1,4)$ in $\mathbb{P}^1\times\mathbb{P}^2$, by Lemma \ref{LemmaS_0}. In particular, they meet in at most $4$ points. Two of those points coincide with $\hat{f}_1\cap E$ and $\hat{f}_2\cap E$. We claim that such intersection consists exactly of those two points. Suppose by contradiction that there is an extra point $p$ in said intersection. Notice that $p\in\mathbb{P}^1\times\{q_{\alpha,\beta}\}$. Then in this case there exist $\alpha_p,\beta_p$ such that $\pi(\sigma(p))$ is in $V(\beta_p f+\alpha_p g)\cap L_0$, and $[\alpha:\beta]\ne[\alpha_p:\beta_p]$. That means that by the same reasoning that we previously applied to the pair $(\alpha,\beta)$, we obtain $p\in \hat{\Lambda}_{\alpha_p,\beta_p}\cap E\simeq\mathbb{P}^1\times \{q_{\alpha_p,\beta_p}\}$. But then $p\in(\mathbb{P}^1\times \{q_{\alpha,\beta}\})\cap (\mathbb{P}^1\times \{q_{\alpha_p,\beta_p}\})$, which is non-empty if and only if $q_{\alpha,\beta}=q_{\alpha_p,\beta_p}$, which is equivalent to $[\alpha:\beta]=[\alpha_p:\beta_p]$, a contradiction.
		
		In case (b), by analogous arguments we obtain that $\hat{S}_0\cap \hat{\Lambda}_{\alpha,\beta}$ is equal to the transform $\hat{f}_0$ of the fiber $f_0$ above the only point in $V(\beta f+\alpha g)\cap L_0$. Moreover, $\hat{S}_0\cap \hat{\Lambda}_{\alpha,\beta}$ and $E\cap \hat{\Lambda}_{\alpha,\beta}$ meet in exactly one point, which coincides with $\hat{f}_0\cap E$.  
	\end{proof}
	
	\paragraph{Proof of Theorem \ref{TheoremBir}}.

	We finish with a few remarks about the birational geometry of $X^+$ that will complete the proof of Theorem \ref{TheoremBir}.
	
	As we noted in the proof of Proposition \ref{PropX+} (see equation (\ref{EquationM^+})), for every Cartier divisor $M$ and every integral curve $C\ne\mathbf{l}_0$ in $X$, we have
	\begin{align*}
		M^+\cdot C^+=M\cdot C+(M\cdot\mathbf{l}_0)\mu(C)
	\end{align*}
	and 
	\begin{align*}
		M^+\cdot h^+=-M\cdot\mathbf{l}_0
	\end{align*}
	where $h$ is a line in the exceptional plane $\Pi$.
	
	Consider any divisor $D\in |\xi+2H|$ as in Lemma \ref{LemmaE(2)} and $F\in|\xi+H|$ as in Lemma \ref{LemmaE(1)}.
	
	Then we have
	\begin{align*}
		D^+\cdot C=D\cdot C-\mu(C);\qquad D^+\cdot h^+=1.\\
		F^+\cdot C=F\cdot C-2\mu(C);\qquad F^+\cdot h^+=2.
	\end{align*}
	
	In particular, by Lemma \ref{LemmaDnef}, we obtain that $D^+$ is a nef divisor in $X^+$.

	\begin{lemma}\label{LemmaD+}
		In $X$, let $C\ne\mathbf{l}_0$ be an integral curve such that $D\cdot C=\mu(C)$. Then $C^+$ is contained in the effective divisor $F^+$. If $C\subset D$, then $D\cdot C=\mu(C)$ implies $C$ lies in the strict transform of one of the planes containing the conics $\mathcal{C}_1$ and $\mathcal{C}_2$. Moreover, every integral curve contained in one of such transforms that has anticanonical degree one in $X$ satisfies $D\cdot C=\mu(C)$. 
	\end{lemma}
	\begin{proof}
		Let $C\ne\mathbf{l}_0$ be an integral curve such that $D\cdot C=\mu(C)$. Then
		\begin{align*}
			F^+\cdot C^+=F\cdot C-2\mu(C)=F\cdot C-2D\cdot C=(\xi+H)\cdot C-2(\xi+2H)\cdot C=-(\xi+3H)\cdot C=-T\cdot C.
		\end{align*}
		Recall that $T$ is a nef divisor in $X$ supporting the extremal ray $\mathbb{R}_{\ge 0}[\mathbf{l}_0]$. Therefore $F^+\cdot C^+<0$, hence $C^+$ is contained in $F^+$. 
		
		Now suppose $C$ is contained in $D$. Let $\Lambda_1$ and $\Lambda_2$ be the strict transform of the planes containing the conics $\mathcal{C}_1$ and $\mathcal{C}_2$ respectively. Without loss of generality, suppose $C$ is not contained in $\Lambda_2$. Recall that $D|_D\equiv 
		3P-E_1-E_2$ (Lemma \ref{LemmaD}). Then
		\begin{align*}
			\mu(C)=D\cdot C=(2P-E_1)\cdot C +(P-E_2)\cdot C.
		\end{align*} 
		Now, $2P-E_1$ is nef, and $P-E_2\equiv \Lambda_2$, so that $(P-E_2)\cdot C\ge\mu(C)$ by Lemma \ref{LemmaLengthInclusions}. This implies that the equality $D\cdot C=\mu(C)$ holds if and only if $(2P-E_1)\cdot C=0$ and $\Lambda_2\cdot C=\mu(C)$.  
		
		By Lemma \ref{LemmaD}, $(2P-E_1)\cdot C=0$ if and only if $[C]\in\mathbb{R}_{\ge0}\mathit{r}_0+\mathbb{R}_{\ge 0}\mathit{e}_2$. Let $\alpha,\beta\ge 0$ such that $C\equiv \alpha \mathit{r}_0+\beta\mathit{e}_1$. Then
		\begin{align*}
			\Lambda_1\cdot C=(P-E_1)\cdot C=-\alpha\le 0.
		\end{align*}
		If $\alpha>0$, it follows that $C\subset \Lambda_1$. 
		If $\alpha=0$, then $C$ is a fiber of $\pi$ over $\mathcal{C}_2$, so that $D\cdot C=1$. The only case when $\mu(C)=D\cdot C$ corresponds to the choice of $C$ among the fibers of $\pi$ over $\mathcal{C}_2\cap \mathbf{l}_0$, which are also contained in $\Lambda_1$. 
		
		By Lemma \ref{LemmaD}, curves with anticanonical degree one in $X$ contained in $\Lambda_1$ coincide with the strict transforms of lines in $\mathcal{P}_1$ intersecting $\mathcal{P}_1\cap \mathcal{C}_2$ at one point. For such a curve $C$, we have $D\cdot C=0$ and $\mu(C)=0$ because $C\cap \mathbf{l}_0=\emptyset$.  
	\end{proof}
	
	\begin{corollary}\label{CorollaryLambda}
		For every $[\alpha:\beta]\in\mathbb{P}^1$, for every point  $p\in \Lambda_{\alpha,\beta}$ there exists an integral curve $C$ in $\Lambda_{\alpha,\beta}$ with $p\in C$ and $D\cdot C=\mu(C)$.   
	\end{corollary}
	\begin{proof}
		By Lemma \ref{CorollaryDeg1}, for every $p\in\Lambda_{\alpha,\beta}\setminus \pi^{-1}(L_0)$ there exists an integral curve $C\subset\Lambda_{\alpha,\beta}$ with $p\in C$ and $-K_X\cdot C=1$. By Lemma \ref{LemmaD+}, it follows that $D\cdot C=\mu(C)$. If $p\in\mathbf{l}_0$ such that $\pi(p)\notin\mathcal{C}_2$, then there exists a curve $C$ in $\Lambda_{\alpha,\beta}$ such that $p\in C$ and $C$ is the strict transform of a line in $\mathcal{P}_1$ passing through $\pi(p)$. In that case $D\cdot C= (3P-E_1-E_2)\cdot (\mathit{r}-2\mathit{e}_1)=1$ and $\mu(C)=\length(C\cap \mathbf{l}_0)=1$. Otherwise, there exists an integral curve $C$ such that $p\in C$ and $C$ is contracted by $\pi$. In that case, $D\cdot C=1=\mu(C)$.  
	\end{proof}
	
	\begin{proposition}\label{PropositionContractionX+}
		The divisor $D^+$ in $X^+$ is nef and is the supporting divisor for an extremal divisorial contraction of $X^+$ of type $(3,1)$ that on the exceptional divisor is a quadric bundle onto a curve, whose fibers are all smooth except for exactly two of them. The exceptional divisor of that contraction coincides with the transform $F^+$ in $X^+$ of the only effective divisor $F$ in the linear system $|\xi+H|$ in $X$.   
	\end{proposition}	
	\begin{proof}
		As a consequence of \ref{LemmaDnef}, we already noticed that $D^+$ is a nef divisor. By Lemma \ref{LemmaD+} we also know that there exist integral curves $C$ in $X$ such that $D\cdot C=\mu(C)$, i.e. $D^+\cdot C^+=0$, which implies that $D^+$ is not ample. Since $\rho_{X^+}=2$, we deduce that it spans an extremal ray of the nef cone of $X^+$. Let $\phi_{D^+}$ be the extremal contraction induced by a multiple of $D^+$. The contraction cannot be of fiber type, since its fibers are covered by $D^+$-trivial curves, whereas the $D^+$-trivial curves do not cover $X^+$ (they are all contained in the divisor $F^+\subset X^+$, by Lemma \ref{LemmaD+}).
		
		Observe that by Corollary \ref{CorollaryLambda} we deduce that for all $[\alpha:\beta]\in\mathbb{P}^1$, the transforms $\Lambda_{\alpha,\beta}^{+}$ of $\Lambda_{\alpha,\beta}$ in $X^{+}$ are covered by curves $C^+$ with $D^+\cdot C^+=0$, hence are contained in the locus of $\phi_{D^+}$. Now, by Proposition \ref{PropositionBlowup}, the section $\hat{\Lambda}_{\alpha,\beta}\cap E$ is isomorphic to $\mathbb{P}^1\times\{q_{\alpha,\beta}\}$ when seen in $E\simeq\mathbb{P}^1\times\mathbb{P}^2$, where $q_{\alpha,\beta}$ is a point in $\mathbb{P}^2$. This implies that the restriction of $\sigma^{+}$ to $\hat{\Lambda}_{\alpha,\beta}$ coincides with the contraction of the curve $\hat{\Lambda}_{\alpha,\beta}\cap E$. There are two cases:
		
		\begin{itemize}
			\item[(a)] For a general $[\alpha:\beta]\in\mathbb{P}^1$, we get $\Lambda_{\alpha,\beta}\simeq \Bl_{2pts}\mathbb{P}^2$. In this case, the section $\hat{\Lambda}_{\alpha,\beta}\cap E$, when seen in $\hat{\Lambda}_{\alpha,\beta}\simeq \Bl_{2 pts}\mathbb{P}^2$, coincides with the transform of the line passing through the blown-up points. This means that the contraction $\sigma^+\colon \hat{X}\to X^+$ restricts to the contraction of the mentioned $(-1)-$curve in $\hat{\Lambda}_{\alpha,\beta}$, yielding $\Lambda_{\alpha,\beta}^+\simeq \mathbb{P}^1\times\mathbb{P}^1$, which meets $\Pi$ in $\{q_{\alpha,\beta}\}$.
			\item[(b)] For exactly two choices of $[\alpha:\beta]\in\mathbb{P}^1$, we obtain that $\Lambda_{\alpha,\beta}$ (and its transform $\hat{\Lambda}_{\alpha,\beta}$) is isomorphic to the blow-up of a quadric cone at a smooth point. In this case, the section $\hat{\Lambda}_{\alpha,\beta}\cap E$, when seen in $\hat{\Lambda}_{\alpha,\beta}$, coincides with the exceptional curve of said blow-up, so that the image $\Lambda_{\alpha,\beta}^+$ of the restriction of $\sigma^{+}$ to $\hat{\Lambda}_{\alpha,\beta}$ is isomorphic to a quadric cone $\mathbb{P}(1,1,2)$ meeting $\Pi$ in $\{q_{\alpha,\beta}\}$.  
		\end{itemize}  
		
		Moreover, again from Proposition \ref{PropositionBlowup}, we have $\underset{[\alpha:\beta]\in\mathbb{P}^1}{\bigcup}\{q_{\alpha,\beta}\}=\sigma^{+}(\hat{S}_0\cap E)$ and $\sigma^{+}|_{\hat{S}_0\cap E}$ is a finite map onto its image, with degree two. Since $\hat{S}_0\cap E$ is a curve in $E\simeq\mathbb{P}^1\times\mathbb{P}^2$ with bidegree $(1,4)$, we deduce that $\sigma^{+}(\hat{S}_0\cap E)$ is a curve of degree equal to $\frac{4}{\deg(\sigma^{+}|_{\hat{S}_0\cap E})}=2$ in $\sigma^{+}(E)=\Pi\simeq\mathbb{P}^2$.
		
		Now, if $\phi_{D^+}$ were a small contraction, its locus would be the disjoint union of finitely many copies of $\mathbb{P}^2$ in $X^+$, by Lemma \ref{LemmaSQMFano}. However, we proved that there are infinitely many distinct copies of $\mathbb{P}^1\times\mathbb{P}^1$ in such locus, so $\phi_{D^+}$ cannot be small, and must therefore be divisorial. Also, observe that in case $(a)$, the two rulings of $\Lambda_{\alpha,\beta}^+$ are obtained as the images through $\sigma^{+}$ of two classes of curves that are distinct in $\NE(\hat{\Lambda}_{\alpha,\beta})$ (they are the two $(-1)$-curves above the blown-up points in $\Bl_{2 pts}\mathbb{P}^2$), but coincide in $\NE(\hat{X})$, because they are the transforms in $\hat{X}$ of two fibers of the $\mathbb{P}^1$-bundle $X\to \mathbb{P}^3$ meeting $\mathbf{l}_0$ transversally in one point. This implies that $\Lambda_{\alpha,\beta}^+$ are two-dimensional fibers of $\phi_{D^+}$. Recall that the surfaces $\Lambda_{\alpha,\beta}$ pairwise intersect only along $\mathbf{l}_0$, so that their transforms $\Lambda_{\alpha,\beta}^+$ in $X^+$ are pairwise disjoint and are all contained in the effective divisor $F^+$ (see Lemma \ref{LemmaD+} and Corollary \ref{CorollaryLambda}). Moreover, since the $\Lambda_{\alpha,\beta}$ sweep $F$ (recall Lemma \ref{LemmaHDHF}), their transforms $\Lambda_{\alpha,\beta}^+$ sweep $F^+$.
		Thus $F^+$ is covered by two-dimensional fibers of $\phi_{D^+}$. Since $F^+$ is a divisor in the 4-fold $X^+$, it follows that $\phi_{D^+}(F^+)$ is either a curve or a point. Notice that, since $F$ contains $S_0$, its transform $F^+$ contains the curve $\sigma^+(\hat{S}_0\cap E)$, that is a conic in $\Pi$, as noted above. Its intersection number with $D^+$ is therefore equal to $D^{+}\cdot 2h^+=2>0$. In particular not all curves in $F^+$ are numerically equivalent, hence $\phi_{D^+}(F^+)$ must be a curve, $\phi_{D^+}$ has infinitely many two-dimensional fibers isomorphic to $\mathbb{P}^1\times\mathbb{P}^1$, and exactly two fibers isomorphic to singular quadric cones.
	\end{proof}
	
	\begin{remark}
		By the classification of contractions of type $(3,1)$ by Takagi, we notice that $\phi_{D^+}$ must fall into the case \cite[Main Theorem (4q)]{Tak99}, i.e it is divisorial and its restriction to the exceptional divisor is a quadric bundle onto a smooth curve $C$ in a variety $Y$, which is analytically locally a hypersurface in $\mathbb{C}^5$, such that $C$ is locally a complete intersection curve in $\mathbb{C}^5$. 
	\end{remark}
	
	\begin{corollary}
		We have the following complete description of the cone of divisors in $X$:
		\begin{align*}
			\Nef(X)&=\mathbb{R}_{\ge 0}[H]+\mathbb{R}_{\ge 0}[\xi+3H].\\
			\Mov(X)&=\mathbb{R}_{\ge 0}[H]+\mathbb{R}_{\ge 0}[\xi+2H].\\
			\Eff(X)&=\mathbb{R}_{\ge 0}[H]+\mathbb{R}_{\ge 0}[\xi+H].
		\end{align*}
	\end{corollary}
	\begin{proof}
		It follows directly from Lemma \ref{LemmaNE} and Proposition \ref{PropositionContractionX+}. 
	\end{proof}

	Università di Torino, Dipartimento di Matematica, via Carlo Alberto 10, 10123 Torino - Italy\\
	Email address: pierroberto.pastorino@unito.it
\end{document}